\documentclass{article}

\usepackage[english]{babel}
\usepackage[utf8]{inputenc}
\usepackage{amsfonts}
\usepackage{latexsym}
\usepackage{amsmath}
\usepackage{amsthm}
\usepackage{amssymb}
\usepackage{mathabx}
\usepackage{enumerate}
\usepackage{csquotes}
\usepackage[backend=biber,
            style=alphabetic,
            giveninits=true,
            isbn=false,
            doi=false,
            url=false,
            sorting=nyvt]{biblatex}
\renewbibmacro{in:}{}

\newtheorem{theorem}{Theorem}
\newtheorem{lemma}[theorem]{Lemma}
\newtheorem{corollary}[theorem]{Corollary}
\newtheorem{proposition}[theorem]{Proposition}
\newtheorem{lettertheorem}{Theorem}

\theoremstyle{definition}

\theoremstyle{remark}

\numberwithin{equation}{section}

\newcommand{\set}[1]{\left\{{#1}\right\}}
\newcommand{\measure}[1]{\mathrm{m} \left({#1}\right)}
\newcommand{\hausdorff}[2]{\mathcal{H}^{#1}\left({#2}\right)}
\newcommand{\dimension}[1]{\mathrm{dim}\left({#1}\right)}
\newcommand{\support}[1]{\mathrm{supp}\left({#1}\right)}
\newcommand{\mathdm}{\mathrm{dm}}
\newcommand{\cF}{\mathcal{F}}
\newcommand{\cG}{\mathcal{G}}

\newcommand{\e}{\varepsilon}
\newcommand{\f}{\varphi}
\newcommand{\disk}{\mathbb{D}}
\newcommand{\ddisk}{\partial\disk}

\newcommand{\complex}{\mathbb{C}}
\newcommand{\re}[1]{\mathrm{Re} \left({#1}\right)}
\newcommand{\Lp}[1]{\mathrm{L}^{#1}}
\newcommand{\wt}{\widetilde}
\newcommand{\ci}[1]{_{ {}_{\scriptstyle #1}}}

\newcommand{\Addresses}{{
  \bigskip
  \footnotesize

  \noindent Spyridon~Kakaroumpas\\
  \textsc{Julius-Maximilians-Universität Würzburg,\\
          Campus Hubland Nord,\\
          Emil-Fischer-Stra{\ss}e 41,\\
          97074 W\"{u}rzburg, Germany}\\
  \textit{E-mail address: } \texttt{spyridon.kakaroumpas@uni-wuerzburg.de}

  \bigskip

  \noindent Odí~Soler~i~Gibert\\
  \textsc{Universitat Politècnica de Catalunya - BarcelonaTech (UPC),\\
          Pavelló I,\\
          Diagonal 647,\\
          08028 Barcelona, Catalunya}\\
  \textit{E-mail address: }\texttt{odi.soler@upc.edu}

}}

\begin{document}

    \title{Preimages under linear combinations of iterates of finite Blaschke products}

    \author{Spyridon~Kakaroumpas\thanks{First author is supported by Alexander von Humboldt Stiftung.} \
            and Odí~Soler~i~Gibert\thanks{Second author is supported in part by the ERC project CHRiSHarMa no.~DLV-682402,
            in part by the Generalitat de Catalunya (grant 2021 SGR 00071),
            the Spanish Ministerio de Ciencia e Innovaci\'on (project  PID2021-123151NB-I00)
            and the Spanish Research Agency (Mar\'ia de Maeztu Program CEX2020-001084-M).}}

    \date{}

    \maketitle

    \begin{abstract}
        Consider a finite Blaschke product $f$ with $f(0) = 0$
        which is not a rotation and denote by $f^n$ its $n$-th iterate.
        Given a sequence $\{a_n\}$ of complex numbers, consider the series $F(z) = \sum_n a_n f^n(z).$
        We show that for any $w \in \complex,$ if $\{a_n\}$ tends to zero but $\sum_n |a_n| = \infty,$
        then the set of points $\xi$ in the unit circle for which the series $F(\xi)$ converges to $w$
        has Hausdorff dimension $1.$
        Moreover, we prove that this result is optimal in the sense that
        the conclusion does not hold in general if one considers Hausdorff measures
        given by any measure function more restrictive than the power functions $t^\delta,$ $0 < \delta < 1.$
    \end{abstract}

    \section{Introduction}
    Consider a lacunary power series with Hadamard gaps,
    that is a power series of the form
    \begin{equation*}
        F(z) = \sum_{n=1}^\infty a_n z^{k_n},
    \end{equation*}
    where $\{a_n\} \subseteq \complex$ and $\{k_n\}$ is a sequence of positive integers satisfying the property
    $\liminf_{n} k_{n+1}/k_{n} > 1.$
    Weiss proved \cite{ref:WeissTheoremOfPaley} a theorem of Paley which states that if $\{a_n\}$ is such that
    $\sum_n |a_n|$ is divergent but $a_n \to 0$ as $n \to \infty,$ then for any point $w \in \complex$
    there exists some point $\xi \in \ddisk$ for which the series $\sum a_n \xi^{k_n}$ is convergent with value $w.$
    In a similar direction, Salem and Zygmund \cite{ref:SalemZygmundLacunarySeriesPeanoCurves} proved that
    boundary values of certain lacunary series are Peano curves.
    This result was refined later on by Kahane, Weiss and Weiss \cite{ref:KahaneWeissWeissLacunarySeries}.
    For more recent results on this topic, see
    \cite{ref:BaranskiLacunarySeries,ref:Belov,ref:Murai1981,ref:Murai1982,ref:Murai1983,ref:YounsiPeanoCurves}.

    Consider now a finite Blaschke product $f$ such that $f(0) = 0,$ that is a function of the form
    \begin{equation*}
        f(z) = \alpha z^m \prod_{n=1}^N \frac{z-z_n}{1-\overline{z_n}z}, \qquad z \in \disk,
    \end{equation*}
    where $\alpha$ is a unimodular constant,
    $m \geq 1$ is the multiplicity of the zero at the origin
    and $\{z_n\}_{n=1}^N$ is a finite possibly empty set of points in $\disk$
    (if it is the empty set, we ignore the rightmost product).
    We will denote by $f^n$ the $n$-th iterate of $f$ on the unit disk,
    that is $f^n = f \circ \dots \circ f.$
    In this article, we will focus on series of the form
    \begin{equation}
     \label{eq:IteratesSeries}
     F(z) = \sum_{n=1}^\infty a_n f^n(z), \qquad z \in \disk,
    \end{equation}
    where $f$ is a finite Blaschke product fixing the origin which is not a rotation
    and $\{a_n\} \subseteq \complex.$
    Recent results show that linear combinations of the form \eqref{eq:IteratesSeries}
    behave similarly to lacunary series (see for instance \cite{ref:NicolauSolerCLT,ref:NicolauConvergenceIterates,ref:DonaireNicolau};
    some of these results hold more generally for iterates of arbitrary inner functions).
    In particular, Donaire and Nicolau showed \cite{ref:DonaireNicolau} the following version of Paley--Weiss Theorem for finite Blaschke products.
    \begin{lettertheorem}[Donaire--Nicolau, \cite{ref:DonaireNicolau}]
        \label{thm:DonaireNicolauPaleyWeissThm}
        Let $f$ be a finite Blaschke product with $f(0) = 0$ which is not a rotation.
        Consider a sequence $\{a_n\}$ of complex numbers tending to $0$ and such that $\sum_n |a_n| = \infty.$
        Then for any $w \in \complex$ there exists $\xi \in \ddisk$ such that $\sum_n a_n f^n(\xi)$ converges
        and $\sum_n a_n f^n(\xi) = w.$
    \end{lettertheorem}

    The main goal of this study is to extend Theorem~\ref{thm:DonaireNicolauPaleyWeissThm} to a quantitative version.
    To this end, we recall some basic definitions and facts about
    Hausdorff measure and dimension.
    For a given arc $I \subseteq \ddisk,$ we denote its normalised length by $\measure{I}.$
    We will consider a measure function to be a function $\varphi\colon [0,+\infty) \longrightarrow [0,+\infty)$ with $\varphi(0) = 0$ which is strictly increasing and continuous from the right.
    Given a measure function $\varphi$ and a subset $E \subseteq \ddisk,$
    we define its $\varphi$-Hausdorff (outer) measure as
    \begin{equation*}
     \hausdorff{\varphi}{E}
     = \lim_{\delta \to 0^+} \inf \left\{\sum_{n=1}^\infty \varphi(\measure{I_n})\colon
     E\subseteq\bigcup_{n=1}^\infty I_n,\,
     \measure{I_n} < \delta \textrm{ for } n \geq 1\right\},
    \end{equation*}
    where the infimum is taken over all countable covers $\lbrace I_n\rbrace$ of $E$ by arcs of length less than $\delta.$
    If we consider two different measure functions $\varphi_1$ and $\varphi_2$
    that satisfy $\lim_{t \to 0^+} \varphi_1(t)/\varphi_2(t) = 0,$
    we will say that $\varphi_1$ is \emph{more restrictive} than $\varphi_2$
    to emphasize the idea that $\hausdorff{\varphi_2}{E} = 0$
    implies $\hausdorff{\varphi_1}{E} = 0$ given any subset $E \subseteq \ddisk.$
    It is a standard fact that the $\sigma$-algebra of $\varphi$-Hausdorff measurable sets
    includes the collection of all Borel sets
    (see for instance~\cite[Section~11.2]{ref:FollandRealAnalysis}).
    Observe that if one takes the measure function $\varphi(t) = t,$
    then the $\varphi$-Hausdorff measure coincides with the normalised Lebesgue measure.
    We will be particularly interested in measure functions of the form
    $\varphi(t) = t^s,$ where $s > 0.$
    In this case, for given $s$ and Borel set $E,$
    we will talk about the $s$-Hausdorff measure of $E$ and denote it by
    $\hausdorff{s}{E}.$
    We define the Hausdorff dimension of a Borel set $E,$
    which we will denote by $\dimension{E},$ as
    \begin{equation*}
     \dimension{E} \coloneq \sup \{s > 0\colon \hausdorff{s}{E} > 0\}.
    \end{equation*}
    We can now state our main result.

    \begin{theorem}
        \label{thm:PaleyWeissDimensionThm}
        Let $f$ be a finite Blaschke product with $f(0) = 0$ which is not a rotation.
        Consider a sequence $\{a_n\}$ of complex numbers tending to $0$ and such that $\sum_n |a_n| = \infty.$
        Then for any $w \in \complex,$ the set
        \begin{equation*}
            A(w) = \set{\xi \in \ddisk\colon \sum_{n=1}^\infty a_n f^n(\xi) \text{ converges and } \sum_{n=1}^\infty a_n f^n(\xi)=w}
        \end{equation*}
        has Hausdorff dimension $1.$
    \end{theorem}

    Observe that the conditions on the sequence $\{a_n\}$ are necessary.
    Indeed, if the sequence does not tend to zero, since for any $n \geq 1$
    it holds that $|f^n(\xi)| = 1$ for every $\xi \in \ddisk,$
    one cannot expect to have any kind of convergence
    for the series \eqref{eq:IteratesSeries} on $\ddisk.$
    Moreover, if the sequence is absolutely summable, then the series
    \eqref{eq:IteratesSeries} is uniformly bounded and one cannot get
    convergence to all $w \in \complex.$
    Theorem~\ref{thm:PaleyWeissDimensionThm} will be a consequence of the following more general statement.
    We postpone the necessary definitions until Section~\ref{sec:Construction}.

    \begin{theorem}
     \label{thm:GeneralPositiveExample}
     Let $f$ be a finite Blaschke product with $f(0) = 0$ which is not a rotation
     and consider a sequence $\{a_n\}$ of complex numbers tending to $0$
     such that $\sum_n |a_n| =  \infty.$
     Consider as well a measure function $\varphi$
     such that $\lim_{t \to 0} \varphi(t)/t = \infty.$
     Then, there exist a constant $0 < c = c(f) < 1$
     and a positive integer $Q = Q(f)$ for which the following holds.

     Suppose that there exist a constant $0 < \beta < 1$
     and a large enough positive integer $N$ such that
     \begin{equation}
      \label{eq:FirstInequalityN}
      \frac{c\beta}{2\sqrt{3}} \left(1+\frac{2}{N-1}\right)^{-1} - \frac{2}{N-1} > 0,
     \end{equation}
     \begin{equation}
      \label{eq:SecondInequalityN}
      \frac{c}{2\sqrt{3(NQ+1)}} \left(1+\frac{2}{N-1}\right)^{-1} - \frac{2}{N-1} > 0
     \end{equation}
     and such that
     \begin{equation}
      \label{eq:LongBlockTendingToZero}
      \lim_{k\to\infty} \sum_{n=\overline{M}_k}^{\overline{N}_k-1} |a_n| = 0
     \end{equation}
     and
     \begin{equation}
      \label{eq:LongBlockReverseCauchy}
      \left(\sum_{n=\overline{M}_k}^{\overline{N}_k-1} |a_n|^2\right)^{1/2}
      \geq \beta \sum_{n=\overline{M}_k}^{\overline{N}_k-1} |a_n|, \qquad k \geq 1,
     \end{equation}
     where the sequences $\{(\overline{M}_k,\overline{N}_k)\}(\varphi,f,Q,N)$
     are defined by~\eqref{eq:SequencesGeneralCase}.
     Then, for any fixed $w \in \complex,$ the set
     \begin{equation*}
            A(w) = \set{\xi \in \ddisk\colon \sum_{n=1}^\infty a_n f^n(\xi) \text{ converges and } \sum_{n=1}^\infty a_n f^n(\xi)=w}
     \end{equation*}
     satisfies that $\hausdorff{\varphi}{A(w)} > 0.$
    \end{theorem}

    The proof of Theorem~\ref{thm:GeneralPositiveExample} will be based
    on a delicate construction originally due to Weiss \cite{ref:WeissTheoremOfPaley}
    that was also used by Donaire and Nicolau \cite{ref:DonaireNicolau}.
    In particular, we start by refining the latter to obtain a Cantor like set $E \subseteq A(w).$
    Then, we show that $\hausdorff{\varphi}{E} > 0$ using standard arguments.

    Observe that Theorem~\ref{thm:PaleyWeissDimensionThm} states that,
    for every $w \in \complex,$
    $\hausdorff{s}{A(w)} > 0$ for any $0 < s < 1,$
    given any finite Blaschke product $f$ fixing the origin which is not a rotation
    and given any complex sequence $\{a_n\}$ tending to zero which is not absolutely summable.
    On the other hand, Theorem~\ref{thm:GeneralPositiveExample} implies that,
    if one considers a measure function $\varphi$ which is more restrictive
    than any power function $\varphi_s(t) = t^s,$ $0 < s < 1$
    (but still less restrictive than $\varphi_1(t) = t$),
    then there are still plenty of examples of finite Blaschke products and sequences
    $\{a_n\}$ for which the set $A(w)$ has $\hausdorff{\varphi}{A(w)} > 0$
    for all $w \in \complex.$
    Indeed, given such a measure function $\varphi$
    and a fixed finite Blaschke product $f$ that fixes the origin
    and that is not a rotation,
    one can easily choose $\beta,$ $N$ and $\{a_n\}$ satisfying \eqref{eq:FirstInequalityN}--\eqref{eq:LongBlockReverseCauchy}.
    For instance, take $\beta = 1/2,$ $N$ large enough so that both \eqref{eq:FirstInequalityN} and \eqref{eq:SecondInequalityN} are satisfied
    and let $a_{n} = 1/k$ if $n = \overline{M}_k$ for some $k \geq 1,$ while $a_n = 0$ otherwise.

    Concerning the proof of Theorem~\ref{thm:PaleyWeissDimensionThm},
    in Section~\ref{sec:ProofMainResults} we will show that when one considers
    the measure function $\varphi_s$ for some $0 < s < 1,$
    any finite Blaschke pro-duct $f$ and any complex sequence $\{a_n\}$
    satisfying the hypotheses of Theorem~\ref{thm:PaleyWeissDimensionThm},
    then there exist automatically $\beta$ and $N$
    for which \eqref{eq:FirstInequalityN}--\eqref{eq:LongBlockReverseCauchy} hold.
    Hence, Theorem~\ref{thm:GeneralPositiveExample} implies that the set $A(w)$
    will have $\hausdorff{s}{A(w)} >0$ for any $w \in \complex$ and,
    since $s < 1$ is arbitrary, $A(w)$ has dimension $1.$

    Finally, it is important to mention that Theorem~\ref{thm:PaleyWeissDimensionThm}
    is optimal in terms of measure functions.
    In other words, if one substitutes the scale of measure functions $\varphi_s,$ $0 < s < 1,$
    by a measure function $\varphi$ more restrictive
    than any power function $\varphi_s$ (but still less restrictive than $\varphi_1$),
    then there are examples for which the set $A(w)$
    has $\hausdorff{\varphi}{A(w)} = 0$ for all $w \in \complex.$
    This is expressed in precise terms in the following theorem,
    which is proved using martingales.
    \begin{theorem}
        \label{thm:Optimality}
        Consider a measure function $\varphi$ such that
        \begin{equation*}
            \lim_{t \to 0^+} \frac{\varphi(t)}{t^s} = 0
        \end{equation*}
        for any $0 < s < 1.$
        Then, there exist a finite Blaschke product $f$ and
        a sequence $\{a_n\} = \{a_n\}(\varphi,f)$ of complex numbers tending to zero with
        $\sum_n |a_n| = \infty$ such that for any $w \in \complex,$ the set
        \begin{equation*}
            A(w) \coloneq \left\{\xi \in \ddisk\colon \sum_{n=1}^\infty a_n f^n(\xi)
            \text{ converges and } \sum_{n=1}^\infty a_n f^n(\xi) = w\right\}
        \end{equation*}
        has Hausdorff measure $\hausdorff{\varphi}{A(w)} = 0.$
    \end{theorem}

    This article is structured as follows.
    In Section~\ref{sec:Preliminary} we collect mostly without proofs some previous results already used
    by Donaire and Nicolau for their Theorem~\ref{thm:DonaireNicolauPaleyWeissThm}.
    Section~\ref{sec:Construction} contains the details of the iterative construction
    used to show the main results.
    In Section~\ref{sec:ProofMainResults} we show both Theorems~\ref{thm:PaleyWeissDimensionThm}
    and~\ref{thm:GeneralPositiveExample}.
    Finally,  in Section~\ref{sec:Optimality} we show Theorem~\ref{thm:Optimality}.

    \subsection*{Acknowledgements}
    The authors would like to thank Artur~Nicolau for helpful discussions on the topic
    and for providing some of the references used in the article.
    They are also thankful for the valuable comments of the referee,
    which have greatly improved the quality of the paper and allowed for the correction of some errors.

    \subsection{Notation and conventions}
    Here we denote the normalised Lebesgue measure on the unit circle $\ddisk$ by $\mathrm{m},$
    so that given a measurable set $E \subseteq \ddisk,$ its normalised Lebesgue measure is $\measure{E}.$
    Sometimes we will also refer to the measure $\measure{E}$ of a set $E \subseteq \ddisk$ as its \emph{length}.
    For $z \in \disk,$ we denote the Poisson kernel by
    \begin{equation*}
        P_z(\xi) = \frac{1-|z|^2}{|\xi-z|^2}, \qquad \xi \in \ddisk
    \end{equation*}
    and, for a measurable set $E \subseteq \ddisk,$ its harmonic measure from point $z$ is defined as
    \begin{equation*}
        w(z,E) \coloneq \int_E P_z(\xi)\,\mathdm(\xi).
    \end{equation*}
    Given a point $z \in \disk,$ we denote by $I(z)$ the closed arc on $\ddisk$ centered at $z/|z| \in \ddisk$
    and of length $\measure{I(z)} = 1-|z|$ (with the convention that $I(0) = \ddisk$).
    Also, given an arc $I \subsetneq \ddisk$ with center $\xi,$ we define $z(I) = (1-\measure{I}) \xi,$
    and if $I = \ddisk,$ then we define $z(I) = 0.$
    Note that with this notation, for any $z \in \disk$ we have that $z(I(z)) = z,$
    while for any closed arc $I \subseteq \ddisk$ it holds that $I(z(I)) = I.$
    Moreover, given $c >0$ and a closed arc $I \subsetneq \ddisk,$ if $c \cdot \measure{I} \leq 1,$
    we define the closed arc $cI$ to be concentric to $I$ and of length $\measure{cI} = c \cdot \measure{I}.$
    In the case in which $c \cdot \measure{I} > 1,$ we just take $cI = \ddisk.$

    \section{Preliminary results}
    \label{sec:Preliminary}
    Here we collect some auxiliary results that will be needed
    in the rest of the paper.
    First, we state without a proof a lemma due to Nicolau (see \cite[Lemma~3.3]{ref:NicolauConvergenceIterates}).

    \begin{lemma}[\cite{ref:NicolauConvergenceIterates}]
        \label{lemma:ArturLemma}
        Consider an inner function $f$ (not necessarily a finite Blaschke product) such that
        it is not a rotation and $f(0) = 0.$
        Then, there exist constants $0 < \e = \e(f) < 1$ and $0 < c = c(f) < 1,$
        such that for all positive integers $M,N$ with $M < N,$
        for all $z \in \disk$ with $|f^M(z)| < \e,$
        and for every sequence $\{a_n\}_{n=M}^N$ of complex numbers, the Borel subset of $\ddisk$ defined by
        \begin{equation}
            \label{eq:ArturLemmaSet}
            C \coloneq \set{\xi \in \ddisk\colon \re{\sum_{n=M}^{N} a_nf^n(\xi)} \geq c \left(\sum_{n=M}^{N}|a_n|^2\right)^{1/2}}
        \end{equation}
        has harmonic measure $w(z,C) \geq c$ with respect to point $z$.
    \end{lemma}

    Using the properties of the Poisson kernel, the previous lemma admits a localized version.

    \begin{corollary}
        \label{corl:LocalisedArturLemma}
        Consider an inner function $f$ (not necessarily a finite Blaschke product) such that
        it is not a rotation and $f(0) = 0.$
        Then, there exist constants $0 < \e = \e(f) < 1$, $0 < c = c(f) < 1$ and a positive integer $d=d(f)$,
        such that for all positive integers $M,N$ with $M < N,$
        for all $z \in \disk$ with $|f^M(z)| < \e,$
        and for every sequence $\{a_n\}_{n=M}^N$ of complex numbers, the Borel subset of $\ddisk$ defined by
        \begin{equation*}
            D \coloneq \set{\xi \in dI(z)\colon \re{\sum_{n=M}^{N} a_nf^n(\xi)} \geq c \left(\sum_{n=M}^{N}|a_n|^2\right)^{1/2}}
        \end{equation*}
        has Lebesgue measure
        \begin{equation*}
            \measure{D} \geq c (1-|z|).
        \end{equation*}
    \end{corollary}
    \begin{proof}
        Denote by $c_0=c_0(f)$ the constant appearing in Lemma~\ref{lemma:ArturLemma} and observe that the lemma still holds if one takes a different constant $c \leq c_0.$
        By direct computation one can check that there is a large enough positive integer $d=d(c_0)$ such that the Poisson kernel satisfies
        \begin{equation*}
            \int_{\ddisk\setminus dI(z)} P_z(\xi)\,\mathdm(\xi)=\int_{\ddisk\setminus dI(z)} \frac{1-|z|^2}{|\xi-z|^2}\,\mathdm(\xi) \leq \frac{c_0}{2}.
        \end{equation*}
        Thus, the set $C$ defined by \eqref{eq:ArturLemmaSet} satisfies
        \begin{equation*}
            w(z,C \cap dI(z)) = \int_{C \cap dI(z)} P_z(\xi)\,\mathdm(\xi)
            \geq \frac{c_0}{2}.
        \end{equation*}
        Next, using the estimate
        \begin{equation*}
            P_z(\xi) \leq \frac{1+|z|}{1-|z|}, \qquad\forall\xi \in \ddisk,
        \end{equation*}
        we get that
        \begin{equation*}
            \measure{D} \geq \frac{c_0}{2}\cdot\frac{1-|z|}{1+|z|} \geq \frac{c_0}{4} (1-|z|).
        \end{equation*}
      Thus, setting $c:=c_0/4$ concludes the proof.
    \end{proof}

    In the rest of this section, we collect without proof some results
    due to Donaire and Nicolau \cite{ref:DonaireNicolau}
    that will be necessary for the construction required to prove our main results.
    Observe that if $f$ is a finite Blaschke product fixing the origin
    which is not a rotation, then $\min\{|f'(\xi)|\colon \xi \in \ddisk\} > 1.$
    In other words, the induced map $f\colon \ddisk \rightarrow \ddisk$ at the boundary
    is expanding in the sense that $\measure{f(I)} > \measure{I}$ for any proper
    subarc of $\ddisk.$
    The following lemma quantifies this expanding behaviour.

    \begin{lemma}[{\cite[Lemma~2.1]{ref:DonaireNicolau}}]
     \label{lemma:QuantitativeExpansion}
     Consider a finite Blaschke product $f$ with $f(0) = 0$ which is not a rotation
     and let $K = K(f) = \min\{|f'(\xi)|\colon \xi \in \ddisk\} > 1.$
     Let $N$ be a positive integer and consider an arc $I \subseteq \ddisk$
     with $\measure{f^N(I)} \leq \delta < 1.$
     Then
     \begin{equation*}
      |f^k(\xi) - f^k(\xi')| \leq \delta K^{k-N}
     \end{equation*}
     for any $\xi,\xi' \in I$ and any integer $1 \leq k \leq N.$
    \end{lemma}

    The next two results are direct consequences of the previous lemma.
    They will allow us to control the oscillation of partial sums of \eqref{eq:IteratesSeries}
    on certain well chosen arcs.

    \begin{corollary}[{\cite[Corollary~2.2]{ref:DonaireNicolau}}]
     \label{corl:OscillationControl}
     Consider a finite Blaschke product $f$ with $f(0) = 0$ which is not a rotation.
     Then there exists a constant $c = c(f) > 0$ such that the following holds.
     If $M < N$ are positive integers, $\{a_n\}_{n=M}^N$ a sequence of complex numbers
     and $I \subseteq \ddisk$ an arc with $\measure{f^N(I)} \leq \delta < 1,$
     then
     \begin{equation*}
      \left|\sum_{k=M}^N a_k (f^k(\xi)-f^k(\xi'))\right|
      \leq c\delta \left(\sum_{n=M}^{N}|a_n|^2\right)^{1/2}
     \end{equation*}
     for any $\xi,\xi' \in I.$
    \end{corollary}

    \begin{corollary}[{\cite[Corollary~2.3]{ref:DonaireNicolau}}]
     \label{corl:OscillationControlSequence}
     Consider a finite Blaschke product $f$ with $f(0) = 0$ which is not a rotation.
     Let $\{a_n\}$ be a sequence of complex numbers tending to zero.
     For each positive integer $N,$ consider an arc $I_N \subseteq \ddisk$
     so that $\sup_N \measure{f^N(I_N)} < 1.$
     Then
     \begin{equation*}
      \max \left\{\left|\sum_{k=1}^N a_k (f^k(\xi)-f^k(\xi'))\right|
      \colon \xi,\xi' \in I_N\right\} \longrightarrow 0
     \end{equation*}
     as $N \to \infty.$
    \end{corollary}

    The result that follows relates the behaviour of $f$ on $\ddisk$ and on $\disk.$
    In particular, it relates the sizes of $f^n(z(I))$ and $f^n(I).$
    This is a minor modification of \cite[Lemma~2.4]{ref:DonaireNicolau},
    although the proof is identical.

    \begin{lemma}
     \label{lemma:RelationInteriorBoundary}
     \begin{enumerate}[(a)]
      \item
      \label{item:InteriorToBoundary}
      For any $0 < \gamma < 1,$ there exists $\delta = \delta(\gamma) > 0$
      such that if $f$ is a finite Blaschke product and $z \in \disk$
      satisfies $|f(z)| \leq \gamma,$ then $\measure{f(I(z))} \geq \delta.$
      \item
      \label{item:BoundaryToInterior}
      Consider a finite Blaschke product $f$ with $f(0) = 0$ which is not a rotation.
      Then, given $0 < \delta < 1/4,$ there exists $0 < \gamma = \gamma(f,\delta) < 1$
      such that if $N$ is a positive integer and $I \subseteq \ddisk$ is an arc with
      $\delta \leq \measure{f^N(I)} \leq 4\delta,$ then $|f^N(z(I))| \leq \gamma.$
     \end{enumerate}
    \end{lemma}

    We finish with a lemma that states the locally constant character of the derivatives of iterates of finite Blaschke products.
    The proof of this estimate is contained in \cite[Proof~of~Lemma~2.4]{ref:DonaireNicolau}.
    However, we include it here explicitly for the convenience of the reader.

    \begin{lemma}
        \label{lemma:QuasiConstantDerivative}
        Consider a finite Blaschke product $f$ with $f(0) = 0$ which is not a rotation
        and let $N$ be a positive integer.
        Then there exists a constant $1 < C = C(f) < \infty$ such that,
        for any arc $I \subseteq \ddisk$ with $\measure{f^N(I)} < 1,$
        it holds that
        \begin{equation*}
            |(f^N)'(\xi)| \leq C |(f^N)'(\xi^\ast)|
        \end{equation*}
        for all $\xi, \xi^\ast \in I.$
    \end{lemma}
    \begin{proof}
     Since for any $x,y > 1$ it holds that $|\log x - \log y| \leq |x-y|,$
     we get that
     \begin{equation*}
      \left|\log \frac{|(f^N)'(\xi)|}{|(f^N)'(\xi^\ast)|}\right| =
      \left|\sum_{k=1}^{N-1} \log \frac{|f'(f^k(\xi))|}{|f'(f^k(\xi^\ast))|}\right|
      \leq \sum_{k=1}^{N-1} |f'(f^k(\xi))-f'(f^k(\xi^\ast))|.
     \end{equation*}
     Now, take $C_1 = \max \{|f''(\xi)|\colon \xi \in \ddisk\}.$
     Then, the previous sum is bounded by
     \begin{equation*}
      \sum_{k=1}^{N-1} |f'(f^k(\xi))-f'(f^k(\xi^\ast))| \leq
      C_1 \sum_{k=1}^{N-1} |f^k(\xi)-f^k(\xi^\ast)|.
     \end{equation*}
     Since $\measure{f^N(I)} \leq \delta$ for some $\delta < 1,$ Lemma~\ref{lemma:QuantitativeExpansion}
     yields the desired result with some constant $C > 1$ that only depends on $C_1$
     and $\min \{|f'(\xi)|\colon \xi \in \ddisk\}.$
    \end{proof}

\section{The construction of Cantor like sets}
\label{sec:Construction}

Consider a finite Blaschke product $f$ and a sequence $\{a_n\}$ satisfying
the hypotheses of Theorem~\ref{thm:PaleyWeissDimensionThm}.
Consider as well a measure function $\varphi.$
As mentioned in the introduction, in order to show that for a given $w \in \complex$
the set
\begin{equation*}
 A(w) = \set{\xi \in \ddisk\colon \sum_{n=1}^\infty a_n f^n(\xi) \text{ converges and } \sum_{n=1}^\infty a_n f^n(\xi)=w}
\end{equation*}
has $\hausdorff{\varphi}{A(w)} > 0,$
we will construct a Cantor like set $E \subseteq A(w)$ for which $\hausdorff{\varphi}{E} > 0$
will hold.
The following lemma contains a collection of sufficient conditions for a Cantor like set
to have positive $\varphi$-Hausdorff measure.
This result is just a minor modification of~\cite[Lemma~2]{ref:HungerfordThesis}
(see also~\cite[Section~3]{ref:RohdeBoundaryBehaviourBlochFunctions}
and~\cite[Theorem~10.5]{ref:PommerenkeBoundaryBehaviourConformalMaps} for further reference).

\begin{lemma}
\label{lemma:HausdorffMeasure}
Consider a measure function $\varphi$ such that the function
\begin{equation*}
 \psi(t) \coloneq \frac{\varphi(t)}{t}
\end{equation*}
is decreasing and satisfies $\lim_{t \to 0^+} \psi(t) = \infty.$
Let $(\cF_n)^{\infty}_{n=0}$ be a sequence of nonempty families of pairwise disjoint
nondegenerate closed arcs of $\ddisk$ such that the following hold:
\begin{enumerate}[(i)]
 \item
 For all $n=0,1,2\ldots$ and for all $J\in\cF_n,$
 there exists $I\in\cF_{n+1}$ such that $I\subseteq J.$
 \item
 For all $n=1,2,\ldots$ and for all $I\in\cF_n,$
 there exists $J\in\cF_{n-1}$ such that $I\subseteq J.$
 \item
 \label{item:ArcCoveringRequirement}
 There exists $0 < c < 1$ such that
 \begin{equation*}
  \sum_{\substack{I\in\cF_{n+1}\\I\subseteq J}}\measure{I} \geq c \cdot \measure{J}
 \end{equation*}
 for all $J\in\cF_n$ and for all $n=0,1,2,\ldots.$
 \item
 \label{item:ArcSizeBound}
 For all $n=0,1,2,\ldots$ and for all $J\in\cF_n$
 it holds that
 if $I\in\cF_{n+1}$ with $I\subseteq J,$ then
 \begin{equation*}
  \frac{\measure{I}}{\measure{J}}
  \leq \min\left\{\frac{c}{2},\frac{\psi^{-1}(c^{-(n+1)})}{\psi^{-1}(c^{-n})}\right\}.
 \end{equation*}
\end{enumerate}

Set
\begin{equation*}
 E \coloneq \bigcap_{n=1}^{\infty}\left(\bigcup_{I\in\cF_n}I\right).
\end{equation*}
Then, there holds $\hausdorff{\varphi}{E} > 0.$
\end{lemma}

Note that in this lemma, we talk about $\psi^{-1}(t),$
even though $\psi$ is only decreasing, instead of strictly decreasing,
and it might be only continuous to the right, instead of actually continuous.
Here and in the rest of the paper, by $\psi^{-1}$ we mean the generalised inverse
\begin{equation*}
    \psi^{-1}(t) = \inf \{s > 0\colon \psi(s) > t\}.
\end{equation*}

\begin{proof}
First of all, we will assume without loss of generality that
$\cF_0$ contains a single closed arc $J_0,$ since the general case follows from the same argument.
We will also assume that $\measure{J_0} \leq \psi^{-1}(1),$
since only the asymptotic behaviour of $\varphi(t)$ as $t \to 0$ is relevant.
Let us denote
\begin{equation*}
 E_n = \bigcup_{I\in\cF_n}I
\end{equation*}
for each integer $n \geq 0.$
We construct a sequence of measures $\{\nu_n\}$ such that $\support{\nu_n} = E_n$ for each $n.$
Define $\nu_0$ by $\nu_0(\ddisk \setminus J_0) = 0$ and by
\begin{equation*}
 \nu_0 = \frac{1}{\measure{J_0}} \mathrm{m}
\end{equation*}
restricted to $J_0.$
That is, $\nu_0$ is the probability measure that is equal to a multiple of
the Lebesgue measure restricted on $J_0.$
Next, assume that we have defined $\nu_{n-1}$ for some integer $n \geq 1$ and
let us construct $\nu_n.$
Observe that $\ddisk \setminus E_{n-1} \subseteq \ddisk \setminus E_n$ and set
\begin{equation*}
 \nu_n(\ddisk \setminus E_n) = 0.
\end{equation*}
For each $I \in \cF_n$ define the restriction of $\nu_n$ on $I$ to be
\begin{equation*}
 \nu_n = \frac{\nu_{n-1}(J)}{\sum_{\substack{K\in\cF_n\\K \subseteq J}} \measure{K}} \mathrm{m},
\end{equation*}
where $J$ is the unique arc in $\cF_{n-1}$ such that $I \subseteq J.$
It is clear that $\nu_n,$ $n \geq 1,$ are all probability measures since
for every $J \in \cF_{n-1}$ we have that
\begin{equation*}
 \sum_{\substack{I\in\cF_n\\I \subseteq J}} \nu_n(I) = \nu_{n-1}(J),
\end{equation*}
so that, by an induction argument, one has $\nu_n(\ddisk) = \sum_{I\in\cF_n} \nu_n(I) = \nu_0(J_0) = 1.$
Moreover, for every nonnegative integer $n$ and for each $I \in \cF_n,$ it holds
by condition \eqref{item:ArcCoveringRequirement} that
\begin{equation}
 \label{eq:SequenceGrowthCondition}
 \nu_n(I) \leq \frac{c^{-n}}{\measure{J_0}} \measure{I}.
\end{equation}
A standard limiting argument as in the proof of~\cite[Lemma~2]{ref:HungerfordThesis} allows us to define a probability Borel measure $\nu$
supported on $E$ such that $\nu(A) = \lim_n \nu_n(A)$ for each Borel set $A.$
In addition, since by construction
\begin{equation*}
 \nu_{n+m}(I) = \nu_n(I)
\end{equation*}
for every $n,m \geq 0$ and $I \in \cF_n,$ we have by \eqref{eq:SequenceGrowthCondition} that
\begin{equation}
 \label{eq:LimitGrowthCondition}
 \nu(I) \leq \frac{c^{-n}}{\measure{J_0}} \measure{I}
\end{equation}
for every $I \in \cF_n$ and $n \geq 0.$

Consider now an arbitrary arc $K$ of $\ddisk,$ for which we can assume that $\measure{K}\leq\psi^{-1}(1).$
Then, there exists a nonnegative integer $N$ such that
\begin{equation*}
 \psi^{-1}(c^{-N-1})\leq\measure{K}\leq\psi^{-1}(c^{-N}).
\end{equation*}
Set $\cF(K) \coloneq \{I\in\cF_{N+1}\colon I \cap K \neq \emptyset\}.$
Observe that, by \eqref{item:ArcSizeBound}, for any nonnegative integer $n$
and for each $I \in \cF_n$ we have that
\begin{equation*}
 \measure{I} \leq \psi^{-1}(c^{-n}).
\end{equation*}
Thus
\begin{equation*}
 \measure{I} \leq \psi^{-1}(c^{-N-1}) \leq \measure{K}
\end{equation*}
for all $I \in \cF(K).$
Then, since $K$ is an arc and $\cF(K)$ is a pairwise disjoint family of arcs $I$
with $I \cap K \neq \emptyset$ and $\measure{I} \leq \measure{K},$
one can easily see that $\sum_{I\in\cF(K)} \measure{I} \leq 3\measure{K}.$
Therefore, by this last observation and by \eqref{eq:LimitGrowthCondition}, we get that
\begin{align*}
 \nu(K) &\leq \sum_{I\in\cF(K)}\nu(I) \leq \frac{3}{\measure{J_0}} c^{-N-1}\measure{K}\\
 &\leq \frac{3c^{-1}}{\measure{J_0}} \psi(\measure{K})\measure{K}
 = \frac{3c^{-1}}{\measure{J_0}} \varphi(\measure{K}).
\end{align*}
An application of a minor variation of one of the implications of Frostman's lemma (see e.g.~\cite[Theorem~8.8]{ref:MattilaGeometryOfSetsAndMeasures}) yields then the desired result.
\end{proof}

For the rest of this section, we set
\begin{equation*}
 C_0 \coloneq \min \{|f'(\xi)|\colon \xi \in \ddisk\} > 1.
\end{equation*}
The next lemma will be our basic construction block.

\begin{lemma}
        \label{lemma:AlternativeInductiveStep}
        Consider a finite Blaschke product $f$ which is not a rotation and such that $f(0) = 0.$
        Then, there exist constants $0 < \e = \e(f) < 1,$
        $0 < c = c(f) < 1,$ $0 < \eta = \eta(f) < 1/4,$ $1 < K = K(f)<\infty$
        and a positive integer $d = d(f)$ such that the following holds.
        Let $M,N$ be positive integers with $M < N,$
        consider $z \in \disk$ such that $|f^M(z)| < \e$
        and let $\{a_n\}_{n=M}^N$ be a sequence of complex numbers.
        Then, there exists a nonempty finite family $\cF$ of pairwise disjoint closed subarcs of $dI(z)$ such that
        \begin{equation}
            \label{eq:IndStepRealPartBound}
            \re{\sum_{n=M}^{N} a_n f^n(\xi)} \geq c\left(\sum_{k=M}^{N}|a_k|^2\right)^{1/2},
            \qquad \text{ for every } \xi \in \bigcup_{I\in\cF} I,
        \end{equation}
        \begin{equation}
            \label{eq:IndStepCoveringRequirement}
            \sum_{I\in\cF} \measure{I} \geq \frac{c}{2d}\measure{dI(z)},
        \end{equation}
        \begin{equation}
            \label{eq:IndStepSizeControl}
            \eta \leq \measure{f^N(I)} \leq 4\eta, \qquad \text{ for each } I \in \cF
        \end{equation}
        and
        \begin{equation}
            \label{eq:IndStepGrowthCondition}
            \measure{I} \leq K C_0^{-(N-M)}\measure{dI(z)}, \qquad \text{ for each } I \in \cF.
        \end{equation}
\end{lemma}
\begin{proof}
Let $\e = \e(f),$ $c_1 = c_1(f)$ and $d = d(f)$ be the constants given by Corollary~\ref{corl:LocalisedArturLemma}.
Consider the set
\begin{equation*}
 E \coloneq \left\{\xi\in dI(z)\colon
 \re{\sum_{k=M}^{N}a_kf^{k}(\xi)} \geq c_1\left(\sum_{k=M}^{N}|a_k|^2\right)^{1/2}\right\},
\end{equation*}
which by Corollary~\ref{corl:LocalisedArturLemma} has
$\measure{E} \geq c_1(1-|z|) \geq (c_1/d) \measure{dI(z)}.$

Observe that since $f(0)=0$, by Schwarz's Lemma we have
\begin{equation*}
 |f^{N}(z)| = |f^{N-M}(f^{M})(z)| \leq |f^{M}(z)| < \e.
\end{equation*}
Therefore, by part~\eqref{item:InteriorToBoundary} of Lemma~\ref{lemma:RelationInteriorBoundary}
we have $\measure{f^{N}(I(z))}, \measure{f^{M}(I(z))} \geq \delta_0$ for some $0 < \delta_0 = \delta_0(\e) < 1$
(and this is still true if we substitute $\delta_0$ by $0 < \delta < \delta_0$).
In the rest of the argument, we will assume that $\measure{dI(z)} < 1,$
since the case $\measure{dI(z)} = 1$ follows with minor modifications.
Let now $0 < \delta_1 = \delta_1(f) < \delta_0$ be a small enough constant
that will be determined later.
For brevity we set $J \coloneq dI(z).$

Now we split the arc $J$ into a collection of subarcs $J_1,\ldots,J_L$
for which we have a precise estimate of $\measure{f^N(J_i)},$ $i = 1,\ldots,L.$
Since we assume that $\measure{J} < 1,$ this arc has two endpoints.
Pick a subarc $J_1$ of $J$ having a common endpoint with $J$ and such that
\begin{equation*}
 \measure{f^{N}(J_1)} = \frac{\delta_1}{2}.
\end{equation*}
We explain the general inductive step.
Suppose that we have chosen $J_1,\ldots,J_{k-1}.$
If $$\measure{f^{N}\left(J\setminus(\bigcup_{1\leq i < k}J_i)\right)} < \delta_1/2,$$
set $L = k-1,$ substitute $J_{k-1}$ by $$J_{k-1} \cup J\setminus(\bigcup_{1\leq i < k}J_i)
= J\setminus(\bigcup_{1\leq i < k-1}J_i)$$ and this finishes the splitting.
Otherwise, pick the subarc $J_k$ of $J\setminus(\bigcup_{1\leq i < k}J_i)$
that has a common endpoint with $J_{k-1}$
(that is, $J_k$ and $J_{k-1}$ have disjoint interiors) and that satisfies
\begin{equation*}
 \measure{f^{N}(J_k)} = \frac{\delta_1}{2}.
\end{equation*}
After finitely many steps, we get a finite sequence of arcs $J_1,\ldots,J_L$
that have disjoint interiors, that cover $J$ and for which
\begin{equation*}
 \frac{\delta_1}{2} \leq \measure{f^N(J_i)} \leq \delta_1
\end{equation*}
for each $i = 1,\ldots,L.$
The reason why there are only finitely many of these arcs is that
they all have lengths bounded below by $\delta_1 (\max\{|f'(\xi)|\colon \xi \in \ddisk\})^{-N}/2.$

Now set
\begin{equation*}
 F \coloneq \{1 \leq i \leq L\colon \measure{E \cap J_i} > 0\}.
\end{equation*}
Clearly, $F \neq \emptyset.$
For each $i \in F,$ pick a nondegenerate closed subarc $\wt{J}_i$ of $J_i$ such that
\begin{equation*}
 \wt{J}_i \subseteq \mathrm{Int}(J_i),
\end{equation*}
that
\begin{equation}
 \label{eq:SubArcCoveringE}
 \measure{E\cap\wt{J}_i} \geq \frac{1}{2}\measure{E\cap J_i}
\end{equation}
and that
\begin{equation}
 \label{eq:SubArcLowerSize}
 \measure{f^{N}(\wt{J}_i)} \geq \frac{1}{2}\measure{f^{N}(J_i)}.
\end{equation}

Set
\begin{equation*}
 \cF \coloneq \{\wt{J}_i\colon i \in F\}
\end{equation*}
and let us check that this collection has the desired properties. Set $c:=c(f)=c_1/4$ (this choice is explained below). First, by \eqref{eq:SubArcCoveringE}, we have that
\begin{equation*}
 \sum_{I\in\cF}\measure{I} \geq \frac{1}{2}\sum_{i=1}^{L}\measure{E\cap J_i}
 =\frac{\measure{E}}{2} \geq \frac{c_1}{2}(1-|z|) \geq \frac{c}{2d}\measure{dI(z)},
\end{equation*}
which shows \eqref{eq:IndStepCoveringRequirement}.
Next, we show \eqref{eq:IndStepGrowthCondition}.
Let $I\in\cF$ be arbitrary and observe that
\begin{equation*}
 \measure{f^{M}(I)} \leq \measure{f^{N}(I)} \leq \delta_1
 < \delta_0 \leq \measure{f^{M}(J)}.
\end{equation*}
Thus, we can pick a subarc $\wt{I}$ of $J$ with $I\subseteq\wt{I}$ such that $\measure{f^{M}(\wt{I})} = \delta_1.$
Let $C = C(f) > 1$ be the constant given by Lemma~\ref{lemma:QuasiConstantDerivative}.
Then, we observe that
\begin{align*}
 1 &\geq \frac{\measure{f^{N}(I)}}{\measure{f^{M}(\wt{I})}}
 = \frac{\int_{I}|(f^{N})'(\xi)|\,\mathdm(\xi)}{\int_{\wt{I}}|(f^{M})'(\xi)|\,\mathdm (\xi)}\\
 &= \frac{\int_{I}|(f^{N-M})'(f^{M}(\xi))|\cdot|(f^{M})'(\xi)|\,\mathdm (\xi)}{\int_{\wt{I}}|(f^{M})'(\xi)|\,\mathdm(\xi)}\\
 &\geq C_0^{N-M} \frac{\int_{I}|(f^{M})'(\xi)|\,\mathdm (\xi)}{\int_{\wt{I}}|(f^{M})'(\xi)|\,\mathdm(\xi)}\\
 &\geq C_0^{N-M} \frac{\measure{I}\cdot(\min_{\xi\in \wt{I}}|(f^{M})'(\xi)|)}{\measure{\wt{I}}\cdot(\max_{\xi\in \wt{I}}|(f^{M})'(\xi)|)}\\
 &\geq \frac{C_0^{N-M}}{C} \frac{\measure{I}}{\measure{\wt{I}}}
 \geq \frac{C_0^{N-M}}{C} \frac{\measure{I}}{\measure{J}},
\end{align*}
which proves \eqref{eq:IndStepGrowthCondition}.
Now, since for all $I \in \cF$ we have $I \cap E \neq \emptyset,$
by Corollary~\ref{corl:OscillationControl} we can pick $\delta_1 = \delta_1(f) > 0$
small enough so that
\begin{equation*}
 \re{\sum_{k=M}^{N}a_kf^{k}(\xi)} \geq
 \frac{c_1}{4} \left(\sum_{k=M}^{N}|a_k|^2\right)^{1/2}
\end{equation*}
for all $\xi \in I.$
Thus, with our choice $c = c_1/4$ property \eqref{eq:IndStepRealPartBound} holds.
Finally, because of \eqref{eq:SubArcLowerSize},
if we set $\eta = \eta(f) = \delta_1/4,$ we also get condition \eqref{eq:IndStepSizeControl}.
\end{proof}

Let us stress an important special case of Lemma~\ref{lemma:HausdorffMeasure}.
Assume that $\varphi(t) = t^{1-\delta}$ for some $0 < \delta < 1.$
Then condition~\eqref{item:ArcSizeBound} in that lemma becomes
\begin{equation*}
 \frac{\measure{I}}{\measure{J}} \leq \min\left\{\frac{\tilde{c}}{2},\tilde{c}^{1/\delta}\right\},
\end{equation*}
where $\tilde{c}$ is the constant appearing in conditions~\eqref{item:ArcCoveringRequirement}
and~\eqref{item:ArcSizeBound}.
In the application, $\tilde{c}$ will be equal to the constant $c/(2d)$ from
condition~\eqref{eq:IndStepCoveringRequirement} from~Lemma~\ref{lemma:AlternativeInductiveStep},
so that condition~\eqref{item:ArcCoveringRequirement} in Lemma~\ref{lemma:HausdorffMeasure}
is automatically granted.
Thus, in order to use condition~\eqref{eq:IndStepGrowthCondition},
we will pick sequences $\{(\overline{M}_k,\overline{N}_k)\}$ with
\begin{equation*}
 K C_0^{-(\overline{N}_k-\overline{M}_k)} \leq \min\left\{\frac{c}{4d},\left(\frac{c}{2d}\right)^{1/\delta}\right\}
\end{equation*}
to allow us to apply Lemma~\ref{lemma:HausdorffMeasure}.
In other words, for measure functions of this form
we can take the differences $\overline{N}_k-\overline{M}_k$ to be constant, once one fixes $\delta.$
In fact, this difference depends linearly on $1/\delta.$
For a general gauge function $\varphi$ we will define the sequences $\{(\overline{M}_k,\overline{N}_k)\}$
to assure that
\begin{equation}
\label{eq:IndexDifferenceEstimate}
 K C_0^{-(\overline{N}_k-\overline{M}_k)} \leq \min\left\{\frac{\tilde{c}}{2},\frac{\psi^{-1}(\tilde{c}^{-(k+1)})}{\psi^{-1}(\tilde{c}^{-k})}\right\}
\end{equation}
for all $k,$ for the same $\tilde{c} = \frac{c}{2d}$
and where $\psi(t) = \varphi(t)/t.$
This enforces the concrete dependence of the differences $\overline{N}_k-\overline{M}_k$ on $\varphi$ and $f.$
In particular, observe that the differences $\overline{N}_k-\overline{M}_k$
(and more generally the sequences $\{\overline{N}_k\}$ and $\{\overline{M}_k\}$ themselves)
do not depend in any way on the sequence $\{a_n\}.$

The definition of the sequences $\{(\overline{M}_k,\overline{N}_k)\}$ is as follows.
Fix positive integers $Q$ and $N$.
Consider a fixed finite Blaschke product $f$ fixing the origin and which is not a rotation.
Recall that $C_0 = \min\{|f'(\xi)|\colon \xi \in \ddisk\} > 1.$
Also consider a fixed measure function $\varphi$ and denote $\psi(t) = \varphi(t)/t;$
we assume that $\psi$ is decreasing and satisfies $\lim_{t\to 0^+} \psi(t) = \infty.$
In addition, consider the constants $0 < c = c(f) < 1,$ $1 < K = K(f) < \infty$
and the positive integer $d = d(f)$ given by Lemma~\ref{lemma:AlternativeInductiveStep},
and define $C = c/(2d).$
Define first the sequence $\{\e_k\} = \{\e_k\}(\varphi,f)$ by
\begin{equation}
 \label{eq:EpsilonKs}
 \e_k = \frac{1}{4} \min\left\{C,\frac{\psi^{-1}(C^{-(k+1)})}{\psi^{-1}(C^{-k})}\right\},
\end{equation}
for each positive integer $k$
(cf. condition~\eqref{item:ArcSizeBound} in Lemma~\ref{lemma:HausdorffMeasure}).
Next, define the sequence $\{G_k\} = \{G_k\}(\varphi,f,Q)$ by taking $G_k$ to be
the first positive integer such that
\begin{equation*}
 K C_0^{-G_k} C_0^Q \leq \e_k
\end{equation*}
for each $k \geq 1$ (cf. estimate~\eqref{eq:IndexDifferenceEstimate}).
Finally, to construct the sequences $$\{(\overline{M}_k,\overline{N}_k)\} =
\{(\overline{M}_k,\overline{N}_k)\}(\varphi,f,Q,N),$$ take $\overline{M}_1 = 1$ and define
\begin{equation}
 \label{eq:SequencesGeneralCase}
 \overline{N}_k = \overline{M}_k + G_k, \qquad
 \overline{M}_{k+1} = \overline{N}_k + NQ
\end{equation}
for $k \geq 1.$

Even though the conditions of Theorem~\ref{thm:GeneralPositiveExample}
are stated in terms of $\{(\overline{M}_k,\overline{N}_k)\},$
the application of Lemma~\ref{lemma:AlternativeInductiveStep} will use
the following auxiliary sequences. As before, consider $\varphi,$ $f,$ $Q$ and $N$ fixed
and let $\{a_n\}$ be a nonsummable sequence of complex numbers tending to zero.
For each positive integer $k,$ pick the integer $1 \leq t_k \leq N$ such that the sum
\begin{equation*}
 \sum_{n=\overline{N}_k+(t_k-1)Q}^{\overline{N}_k+t_kQ-1} |a_n|
\end{equation*}
is the minimum.
In particular, observe that once $t_k$ is determined, we have that
\begin{equation}
 \label{eq:ShortBlockSmall}
 \sum_{n=\overline{N}_k+(t_k-1)Q}^{\overline{N}_k+t_kQ-1} |a_n|
 \leq \frac{1}{N} \sum_{n=\overline{N}_k}^{\overline{M}_{k+1}-1} |a_n|.
\end{equation}
We now define the sequences $\{(M_k,N_k)\} = \{(M_k,N_k)\}(\varphi,f,Q,N,\{a_n\})$ by setting
$M_1 = \overline{M}_1 = 1$ and
\begin{equation}
 \label{eq:ConstructionSequences}
 N_k = \overline{N}_k + (t_k-1)Q, \qquad
 M_{k+1} = \overline{N}_k + t_kQ
\end{equation}
for each positive integer $k.$ In particular $M_{k+1}-N_k=Q,$ for all $k\geq1.$

Note that, for this new sequence, we have
\begin{equation*}
N_{k+1}\leq\overline{N}_{k+1}+(N-1)Q,\quad N_{k}\geq \overline{N}_k=\overline{M}_{k+1}-NQ,
\end{equation*}
for all $k\geq1.$ Assuming thus for a moment that in addition condition~\eqref{eq:LongBlockTendingToZero},
which is
\begin{equation*}
\lim_{k\rightarrow\infty}\sum_{n=\overline{M}_k}^{\overline{N}_{k}-1}|a_n|=0,
\end{equation*}
holds, then since $\lim_{n\rightarrow\infty}a_n=0$ and $N,Q$ are fixed numbers we immediately conclude
\begin{equation*}
 \lim_{k\to\infty}\sum_{n=N_k}^{N_{k+1}}|a_n| = 0.
\end{equation*}
In addition, estimate~\eqref{eq:ShortBlockSmall} becomes
\begin{equation*}
 \sum_{n=N_k}^{M_{k+1}-1} |a_n|
 \leq \frac{1}{N} \sum_{n=\overline{N}_k}^{\overline{M}_{k+1}-1} |a_n|,
\end{equation*}
implying that
\begin{align*}
S_k &\coloneq \sum_{n=N_k+1}^{M_{k+1}-1}|a_n| \leq\sum_{n=N_k}^{M_{k+1}-1}|a_n|\leq
\frac{1}{N}\sum_{n=\overline{N}_k}^{\overline{M}_{k+1}-1}|a_n|\\
&=\frac{1}{N}\left(S_{k}+\sum_{n=\overline{N}_k}^{N_k}|a_n|+\sum_{n=M_{k+1}}^{\overline{M}_{k+1}-1}|a_n|\right)\\
&\leq\frac{1}{N}\left(S_{k}+\sum_{n=M_k}^{N_k}|a_n|+\sum_{n=M_{k+1}}^{N_{k+1}}|a_n|\right),
\end{align*}
therefore
\begin{equation*}
 \sum_{n=M_{k}}^{N_{k}} |a_n| + \sum_{n=M_{k+1}}^{N_{k+1}} |a_n|
 \geq (N-1) \sum_{n=N_{k}}^{M_{k+1}-1} |a_n|.
\end{equation*}
It follows that for all positive integers $k_1 < k_2$ there holds
\begin{equation}
 \label{eq:FirstBoundWithLongBlocks}
 \sum_{k=k_1}^{k_2} \sum_{n=M_k}^{N_k} |a_n|
 \geq \frac{N-1}{2} \sum_{k=k_1}^{k_2-1} \sum_{n=N_k+1}^{M_{k+1}-1} |a_n|,
\end{equation}
and therefore
\begin{equation}
 \label{eq:SecondBoundWithLongBlocks}
 \sum_{n=M_{k_1}}^{N_{k_2}} |a_n|
 \leq \left(1+\frac{2}{N-1}\right) \sum_{k=k_1}^{k_2}\sum_{n=M_k}^{N_k} |a_n|.
\end{equation}
Lastly, by the Cauchy-Schwarz inequality, if condition~\eqref{eq:LongBlockReverseCauchy}, which is
\begin{equation*}
\left(\sum_{n=\overline{M}_k}^{\overline{N}_k-1}|a_n|^2\right)^{1/2}\geq\beta\sum_{n=\overline{M}_k}^{\overline{N}_k-1}|a_n|,\quad k=1,2,\ldots
\end{equation*}
happens to be satisfied for some constant $0<\beta<1$, then we obtain that
\begin{align*}
&\left(\sum_{n=M_k}^{N_k}|a_n|^2\right)^{1/2}=\left(\sum_{n=M_k}^{\overline{M}_k-1}|a_n|^2+\sum_{n=\overline{M}_k}^{\overline{N}_k-1}|a_n|^2+\sum_{n=\overline{N}_k}^{N_k}|a_n|^2\right)^{1/2}\\
&\geq\frac{1}{\sqrt{3}}\left[\left(\sum_{n=M_k}^{\overline{M}_k-1}|a_n|^2\right)^{1/2}+\left(\sum_{n=\overline{M}_k}^{\overline{N}_k-1}|a_n|^2\right)^{1/2}+\left(\sum_{n=\overline{N}_k}^{N_k}|a_n|^2\right)^{1/2}\right]\\
&\geq\frac{1}{\sqrt{3}}\left[\frac{1}{\sqrt{NQ+1}}\sum_{n=M_k}^{\overline{M}_k-1}|a_n|+\beta\sum_{n=\overline{M}_k}^{\overline{N}_k-1}|a_n|+\frac{1}{\sqrt{NQ+1}}\sum_{n=\overline{N}_k}^{N_k}|a_n|\right],
\end{align*}
and therefore
\begin{equation}
 \label{eq:CoefficientsReverseCauchy}
 \left(\sum_{n=M_k}^{N_k} |a_n|^2\right)^{1/2}
 \geq \beta_1 \sum_{n=M_k}^{N_k} |a_n|
\end{equation}
for all $k \geq 1$ by taking
\begin{equation}
 \label{eq:BetaDefinition}
 \beta_1 = \frac{1}{2\sqrt{3}} \min\left\{\beta, \frac{1}{\sqrt{NQ+1}}\right\}.
\end{equation}

The following proposition deals with the general construction of our Cantor like sets.
Here, we will need to use the pseudohyperbolic distance in the unit disk.
Recall that, given two points $z,w \in \disk,$ we define their pseudohyperbolic distance
$\rho(z,w)$ by the expression
\begin{equation*}
    \rho(z,w) = \left|\frac{z-w}{1-\overline{w}z}\right|.
\end{equation*}
We refer the reader to~\cite[Section~I.1]{ref:GarnettBoundedAnalyticFunctions}
for the basic facts and properties of the pseudohyperbolic distance.

\begin{proposition}
 \label{prop:LongBlocksConstruction}
 Consider a measure function $\varphi$ with $\psi(t) = \varphi(t)/t$ decreasing and
 with $\lim_{t\to 0^+} \psi(t) = \infty.$
 Consider also a finite Blaschke product $f$ with $f(0) = 0$ which is not a rotation.
 Let $\{a_n\}$ be a nonsummable sequence of complex numbers tending to zero.
 Also let $\{\e_k\} = \{\e_k\}(\varphi,f)$ be defined by~\eqref{eq:EpsilonKs}.

 Then, there exist $0 < \varepsilon = \varepsilon(f) < 1,$
 $0 < C = C(f) < 1$ and positive integers $d = d(f)$ and $Q = Q(f)$ for which the following holds.
 Suppose that there exist $0 < \beta < 1$ and a positive integer $N$ for which
 conditions~\eqref{eq:FirstInequalityN}--\eqref{eq:LongBlockReverseCauchy} hold
 (with $C$ taking the role of $c$ in these conditions)
 and consider the sequences $\{(M_k,N_k)\} = \{(M_k,N_k)\}(\varphi,f,Q,N,\{a_n\})$
 defined by~\eqref{eq:ConstructionSequences}.
 Then, for any positive integers $k_1 \leq k_2$ and $z \in \disk$ such that
 $|f^{M_{k_1}}(z)| < \varepsilon,$ there exists a sequence $\{\cF_k\}_{k=k_1}^{k_2}$
 of families of arcs with the following properties:
 \begin{enumerate}[(i)]
  \item
  \label{item:FirstConditionForFamilies}
  For each $k=k_1,\ldots,k_2,$ $\cF_k$ is a nonempty finite family of pairwise disjoint closed subarcs of $dI(z).$
  \item
  For each $k=k_1,\ldots,k_2-1$ and for each $I\in\cF_{k+1}$ there exists $J\in\cF_{k}$ such that $I\subseteq J.$
  \item
  \label{item:ThirdConditionForFamilies}
  For each $k=k_1,\ldots,k_2-1$ and for each $J\in\cF_{k}$ there exists $I\in\cF_{k+1}$ such that $I\subseteq J.$
  \item
  For each $k=k_1,\ldots,k_2-1$ there holds
  \begin{equation}
   \label{eq:LongBlocksSizeControl}
   \max_{I \in \cF_{k+1}}
   \left\{\frac{\measure{I}}{\measure{I'}}\right\} \leq \e_k,
  \end{equation}
  where $I'$ is the unique arc in $\cF_k$ with $I \subseteq I'.$
  \item
  For each $k=k_1,\ldots,k_2-1$ and for each $J\in\cF_{k}$ there holds
  \begin{equation}
   \label{eq:LongBlocksCovering}
   \sum_{\substack{I\in\cF_{k+1}\\I\subseteq J}} \measure{I} \geq C \cdot \measure{J}.
  \end{equation}
  \item
  There holds
  \begin{equation}
   \label{eq:LongBlocksRealPart}
   \re{\sum_{n=M_{k_1}}^{N_{k_2}}a_nf^n(\xi)} \geq R \sum_{n={M_{k_1}}}^{N_{k_2}}|a_n|,\qquad \text{ for every } \xi\in\bigcup_{I\in\cF_{k_2}}I,
  \end{equation}
  where $R = R(f,\beta,N,Q) > 0.$
 \end{enumerate}
\end{proposition}
\begin{proof}
 Let $\varepsilon = \varepsilon(f)$ be the constant given by Lemma~\ref{lemma:AlternativeInductiveStep}.
 Let also $d = d(f)$ be the positive integer given by the same lemma
 and take $C = C(f) = c/(2d),$ where $0 < c = c(f) < 1$ is also given
 by Lemma~\ref{lemma:AlternativeInductiveStep}.
 First, we determine the value of $Q = Q(f).$
 Since we will have to apply Lemma~\ref{lemma:AlternativeInductiveStep} iteratively,
 consider the constant $0 < \eta = \eta(f) < 1/4$ given in that result.
 By part~\eqref{item:BoundaryToInterior} of Lemma~\ref{lemma:RelationInteriorBoundary},
 there exists a constant $0 < \gamma = \gamma(f,\eta) < 1$ such that if $N$ is any
 positive integer and $I \subseteq \ddisk$ is an arc with
 $\eta < \measure{f^N(I)} < 4\eta,$ then $|f^N(z(I))| \leq \gamma.$
 Later on, we will need to consider points $z^\ast$ that are, at most,
 at a fixed hyperbolic distance (depending on $d$) of $z(I).$
 That is, we will consider $z^\ast$ with $\rho(z(I),z^\ast) < \rho_0 = \rho_0(d) < 1.$
 Therefore, by Schwarz Lemma there exists $0 < \gamma_1 = \gamma_1(\gamma) < 1$ such that
 $|f^N(z^\ast)| \leq \gamma_1$ for any $z^\ast$ with $\rho(z(I),z^\ast) < \rho_0.$
Finally, since by the Denjoy--Wolff Theorem $f^n$ tend to zero uniformly on compact sets,
 we can choose a positive integer $Q = Q(f,\gamma_1) = Q(f)$ large enough so that
 $|f^Q(z)| < \varepsilon$ whenever $|z| \leq \gamma_1.$

 We now assume that $\beta,$ $N$ and $z \in \disk$ have been chosen satisfying
 the required conditions and show how to perform the inductive construction.
 For the first step, we have by assumption $|f^{M_{k_1}}(z)| < \e.$
 Now we use Lemma~\ref{lemma:AlternativeInductiveStep} with the integers $M_{k_1} < N_{k_1}$
 to pick a nonempty finite family $\cF_{k_1}$ of pairwise disjoint closed subarcs of $dI(z)$ such that
 \begin{equation*}
  \re{\sum_{m=M_{k_1}}^{N_{k_1}} a_mf^{m}(\xi)}
  \geq c \left(\sum_{n=M_{k_1}}^{N_{k_1}} |a_n|^2\right)^{1/2},\qquad
  \text{ for every } \xi\in\bigcup_{I\in\cF_{k_1}} I,
 \end{equation*}
 \begin{equation*}
  \sum_{I\in\cF_{k_1}} \measure{I} \geq \frac{c}{2d}\measure{dI(z)},
 \end{equation*}
 \begin{equation*}
  \eta \leq \measure{f^{N_{k_1}}(I)} \leq 4\eta,\qquad
  \text{ for each } I\in\cF_{k_1},
 \end{equation*}
 and
 \begin{equation*}
  \measure{I} \leq K C_0^{-(N_{k_1}-M_{k_1})} \measure{dI(z)},\qquad
  \text{ for each } I\in\cF_{k_1},
 \end{equation*}
 where $1 < K = K(f) < \infty$ is the constant given by Lemma~\ref{lemma:AlternativeInductiveStep}
 and $C_0 = \min\{|f'(\xi)|\colon \xi \in \ddisk\} > 1.$
 If $k_1 = k_2,$ we have finished, so we assume that $k_2 > k_1.$
 Suppose that we have already defined $\cF_{k}$ up to some $k < k_2$
 and consider $J \in \cF_k.$
 Take $z^\ast(J) \in \disk$ such that $dI(z^\ast(J)) = I(z(J)).$
 This can be done by picking $z^\ast(J)$ on the same radius as $z(J)$
 with $d(1-|z^\ast(J)|) = 1-|z(J)|.$
 In particular, we have that
 \begin{equation*}
  \rho(z^\ast(J),z(J)) \leq \frac{d-1}{d} < 1.
 \end{equation*}
 Hence, by~\eqref{eq:IndStepSizeControl} and part (\ref{item:BoundaryToInterior}) of Lemma~\ref{lemma:RelationInteriorBoundary},
 we have that $|f^{N_k}(z(J))| \leq \gamma$ and, by the previous argument,
 also $|f^{N_k}(z^\ast(J))| \leq \gamma_1.$
 By the choice of $Q$ and because $M_{k+1}-N_k = Q,$ it follows that
 $|f^{M_{k+1}}(z)| < \varepsilon.$
 Therefore, we can apply Lemma~\ref{lemma:AlternativeInductiveStep}
 with $z^\ast(J)$ and the integers $M_{k+1} < N_{k+1}$ to obtain
 a nonempty finite family $\cF_{k+1}(J)$ of pairwise disjoint closed subarcs of $J$
 such that
 \begin{equation*}
  \re{\sum_{n=M_{k+1}}^{N_{k+1}} a_nf^{n}(\xi)} \geq
  c \left(\sum_{n=M_{k+1}}^{N_{k+1}} |a_n|^2\right)^{1/2},\qquad
  \text{ for all } \xi\in\bigcup_{I\in\cF_{k+1}(J)} I,
 \end{equation*}
 \begin{equation*}
  \sum_{I\in\cF_{k+1}(J)} \measure{I} \geq \frac{c}{2d} \measure{J},
 \end{equation*}
 \begin{equation*}
  \eta \leq \measure{f^{N_{k+1}}(I)} \leq 4\eta,\qquad
  \text{ for each } I\in\cF_{k+1}(J),
 \end{equation*}
 and
 \begin{equation}
  \label{eq:GenerationSizeControl}
  \measure{I} \leq K C_0^{-(N_{k+1}-M_{k+1})} \measure{J},\qquad
  \text{ for each } I\in\cF_{k+1}(J).
 \end{equation}
 Finally, we set
 \begin{equation*}
  \cF_{k+1} \coloneq \bigcup_{J\in\cF_k}\cF_{k+1}(J).
 \end{equation*}

 Conditions \eqref{item:FirstConditionForFamilies}--\eqref{item:ThirdConditionForFamilies}
 are immediately satisfied due to Lemma~\ref{lemma:AlternativeInductiveStep}.
 Condition~\eqref{eq:LongBlocksSizeControl} is easily verified
 using~\eqref{eq:GenerationSizeControl}.
 Indeed, for any $k_1 \leq k < k_2$ and any $I \in \cF_{k+1},$
 if $I'$ is the unique arc in $\cF_k$ such that $I \subseteq I',$ we have that
 \begin{equation*}
  \frac{\measure{I}}{\measure{I'}} \leq K C_0^{-(N_{k+1}-M_{k+1})} \leq
  K C_0^{-G_k} \leq \e_k.
 \end{equation*}
 Condition~\eqref{eq:LongBlocksCovering} also follows immediately from
 Lemma~\ref{lemma:AlternativeInductiveStep}.
 
 We are only left with checking~\eqref{eq:LongBlocksRealPart}.
 It is clear that for any $k_1 \leq l \leq k_2$ we have that
 \begin{equation*}
  \re{\sum_{n=M_{l}}^{N_{l}} a_nf^{n}(\xi)} \geq
  c \left(\sum_{n=M_{l}}^{N_{l}} |a_n|^2\right)^{1/2},\qquad
  \text{ for every } \xi\in\bigcup_{I\in\cF_{k_2}}I.
 \end{equation*}
 It follows by~\eqref{eq:CoefficientsReverseCauchy} that
 \begin{equation*}
  \re{\sum_{n=M_{l}}^{N_{l}} a_nf^{n}(\xi)} \geq
  c \beta_1 \sum_{n=M_{l}}^{N_{l}} |a_n|,\qquad
  \text{ for every } \xi\in\bigcup_{I\in\cF_{k_2}} I,
 \end{equation*}
 where $\beta_1$ is defined by~\eqref{eq:BetaDefinition}.
 Thus
 \begin{equation}
\label{eq:SummedCoefficientsReverseCauchy}
  \re{\sum_{k=k_1}^{k_2} \sum_{n=M_{k}}^{N_{k}} a_nf^{n}(\xi)} \geq
  c \beta_1 \sum_{k=k_1}^{k_2} \sum_{n=M_{k}}^{N_{k}} |a_n|,
 \end{equation}
for every $\xi\in\bigcup_{I\in\cF_{k_2}}I.$
 Hence, using~\eqref{eq:FirstBoundWithLongBlocks},~\eqref{eq:SecondBoundWithLongBlocks} as well as~\eqref{eq:SummedCoefficientsReverseCauchy},
 it follows that for all $\xi\in\bigcup_{I\in\cF_{k_2}}I$ we have
 \begin{align*}
  \re{\sum_{n=M_{k_1}}^{N_{k_2}} a_nf^{n}(\xi)} &=
  \re{\sum_{k=k_1}^{k_2} \sum_{n=M_{k}}^{N_{k}} a_nf^{n}(\xi)}
  + \re{\sum_{k=k_1}^{k_2-1} \sum_{n=N_{k}+1}^{M_{k+1}-1} a_nf^{n}(\xi)}\\
  &\geq c \beta_1 \sum_{k=k_1}^{k_2} \sum_{n=M_{k}}^{N_{k}} |a_n|
  - \sum_{k=k_1}^{k_2-1} \sum_{n=N_{k}+1}^{M_{k+1}-1} |a_n|\\
  &\geq R \sum_{n=M_{k_1}}^{N_{k_2}}|a_n|
 \end{align*}
 by choosing for instance
 \begin{equation*}
  R = c \beta_1 \left(1+\frac{2}{N-1}\right)^{-1} - \frac{2}{N-1},
 \end{equation*}
 which is positive by conditions~\eqref{eq:FirstInequalityN} and~\eqref{eq:SecondInequalityN}.
\end{proof}

\begin{corollary}
 \label{corl:BuildingBlock}
 Consider a measure function $\varphi$ with $\psi(t) = \varphi(t)/t$ decreasing and
 with $\lim_{t\to 0^+} \psi(t) = \infty.$
 Consider also a finite Blaschke product $f$ with $f(0) = 0$ which is not a rotation.
 Let $\{a_n\}$ be a nonsummable sequence of complex numbers tending to zero.
 Also let $\{\e_k\} = \{\e_k\}(\varphi,f)$ be defined by~\eqref{eq:EpsilonKs}.

 Then, there exist $0 < \varepsilon = \varepsilon(f) < 1,$
 $0 < C = C(f) < 1$ and positive integers $d = d(f)$ and $Q = Q(f)$ for which the following holds.
 Suppose that there exist $0 < \beta < 1$ and a positive integer $N$ for which
 conditions~\eqref{eq:FirstInequalityN}--\eqref{eq:LongBlockReverseCauchy} hold
 (with $C$ taking the role of $c$ in these conditions)
 and consider the sequences $\{(M_k,N_k)\} = \{(M_k,N_k)\}(\varphi,f,Q,N,\{a_n\})$
 defined by~\eqref{eq:ConstructionSequences}.
 Then, there exists $r = r(f,\beta,Q,N) > 0$ such that,
 for any positive integers $k_1 \leq k_2,$
 any $z \in \disk$ such that $|f^{M_{k_1}}(z)| < \varepsilon$
 and any $w \in \complex$ such that
 \begin{equation*}
  \sum_{n=M_{k_1}}^{N_{k_2}} |a_n| \leq r|w|,
 \end{equation*}
 there exists a sequence $\{\cF_k\}_{k=k_1}^{k_2}$
 of families of arcs with the following properties:
 \begin{enumerate}[(i)]
  \item
  For each $k=k_1,\ldots,k_2,$ $\cF_k$ is a nonempty finite family of pairwise disjoint closed subarcs of $dI(z).$
  \item
  For each $k=k_1,\ldots,k_2-1$ and for each $I\in\cF_{k+1}$ there exists $J\in\cF_{k}$ such that $I\subseteq J.$
  \item
  For each $k=k_1,\ldots,k_2-1$ and for each $J\in\cF_{k}$ there exists $I\in\cF_{k+1}$ such that $I\subseteq J.$
  \item
  For each $k=k_1,\ldots,k_2-1$ there holds
  \begin{equation*}
   \max_{I \in \cF_{k+1}}
   \left\{\frac{\measure{I}}{\measure{I'}}\right\} \leq \e_k,
  \end{equation*}
  where $I'$ is the unique arc in $\cF_k$ with $I \subseteq I'.$
  \item
  For each $k=k_1,\ldots,k_2-1$ and for each $J\in\cF_{k}$ there holds
  \begin{equation*}
   \sum_{\substack{I\in\cF_{k+1}\\I\subseteq J}} \measure{I} \geq C \cdot \measure{J}.
  \end{equation*}
  \item
  There holds
  \begin{equation*}
   \left|w - \sum_{n=M_{k_1}}^{N_{k_2}} a_nf^n(\xi)\right| \leq
   |w| - r \sum_{n={M_{k_1}}}^{N_{k_2}} |a_n|,\qquad
   \text{ for every } \xi\in\bigcup_{I\in\cF_{k_2}}I.
  \end{equation*}
 \end{enumerate}
\end{corollary}
\begin{proof}
 The proof uses the same geometric argument than the one applied in
 \cite[Corollary~3.1]{ref:DonaireNicolau} to a single point.
 The only difference is that one has to apply it to all points in
 $\bigcup_{I\in\cF_{k_2}}I$ given by Proposition~\ref{prop:LongBlocksConstruction}.
\end{proof}

\section{Proof of main results}
\label{sec:ProofMainResults}
We begin this section by proving Theorem~\ref{thm:PaleyWeissDimensionThm}
using Theorem~\ref{thm:GeneralPositiveExample}.

\begin{proof}[Proof of Theorem~\ref{thm:PaleyWeissDimensionThm}]
 Let $0 < \delta < 1$ be arbitrary.
 Set $\f(t) = t^{1-\delta}.$
 In particular, note that
 \begin{equation*}
  \psi(t) \coloneq \frac{\varphi(t)}{t} = t^{-\delta}
 \end{equation*}
 is decreasing and $\lim_{t\to 0^+} \psi(t) = \infty.$
 Take $C = c/(2d)$ where the constant $0 < c < 1$ and the positive integer $d$
 are given by Lemma~\ref{lemma:AlternativeInductiveStep} and depend only on $f.$
 If we define $\{\e_k\}$ by~\eqref{eq:EpsilonKs}, we get
 \begin{equation*}
  \varepsilon_k = \frac{1}{4} \min\{C,C^{1/\delta}\} = \frac{C^{1/\delta}}{4}.
 \end{equation*}
 In other words, $\e_k = \e = \e(\varphi,f)$ for all $k \geq 1,$
 so we also get $G_k = G = G(\varphi,f)$ for all $k.$

 Now, we show that for any nonsummable sequence $\{a_n\}$ of complex numbers tending to zero,
 we can choose $0 < \beta < 1$ and a positive integer $N$ for which the requirements
 of Theorem~\ref{thm:GeneralPositiveExample} are satisfied.
 Let $Q$ be the integer given by Theorem~\ref{thm:GeneralPositiveExample}.
 Choose
 \begin{equation*}
  \beta = \frac{1}{\sqrt{G}}
 \end{equation*}
 and pick a positive integer $N$ large enough so that
 conditions~\eqref{eq:FirstInequalityN} and~\eqref{eq:SecondInequalityN} are satisfied.
 Let $\{(\overline{M}_k,\overline{N}_k)\}$ be the sequences defined
 by~\eqref{eq:SequencesGeneralCase}.
 Then, since $\overline{N}_k - \overline{M}_k = G$ for each $k \geq 1,$
 Cauchy--Schwarz inequality and our choice of $\beta$ give
 condition~\eqref{eq:LongBlockReverseCauchy}.
 Moreover, since $|a_n| \longrightarrow 0$ as $n \to \infty,$
 we also get condition~\eqref{eq:LongBlockTendingToZero}.
 Therefore, Theorem~\ref{thm:GeneralPositiveExample} yields that
 for any $w \in \complex,$ we have that $\hausdorff{1-\delta}{A(w)} > 0.$
 The fact that $\delta$ was arbitrary implies that $\dimension{A(w)} = 1,$
 as we wanted to show.
\end{proof}

The rest of this section is devoted to proof Theorem~\ref{thm:GeneralPositiveExample},
which includes the case of a general measure function $\varphi.$

\begin{proof}[Proof of Theorem~\ref{thm:GeneralPositiveExample}]
 Given a finite Blaschke product $f$ with $f(0) = 0$ which is not a rotation,
 consider the constants $0 < \e = \e(f) < 1$ and $0 < C = C(f) < 1$ and
 the integers $d = d(f)$ and $Q = Q(f)$ given by Corollary~\ref{corl:BuildingBlock}.
 Given a sequence $\{a_n\}$ of complex numbers tending to zero but such that
 \begin{equation*}
  \sum_{n=1}^\infty |a_n| = \infty,
 \end{equation*}
 assume that there are $0 < \beta < 1$ and a (large enough) positive integer $N$ so that
 conditions~\eqref{eq:FirstInequalityN}--\eqref{eq:LongBlockReverseCauchy}
 are satisfied.
 Consider also the sequences $\{\e_k\}$ (defined by~\eqref{eq:EpsilonKs}),
 $\{(\overline{M}_k,\overline{N}_k)\}$ (defined by~\eqref{eq:SequencesGeneralCase})
 and $\{(M_k,N_k)\}$ (defined by~\eqref{eq:ConstructionSequences}).
 Lastly, consider a fixed $w \in \complex.$

 Pick $z_0 \in \disk \setminus \{0\}$ such that $|f(z_0)| < \e$ and
 set $J_0 \coloneq dI(z_0)$ and $\cG_0 = \{J_0\}.$
 We define inductively a sequence $(\cG_k)^{\infty}_{k=0}$
 of nonempty finite families of pairwise disjoint closed arcs,
 having the following properties:
 \begin{enumerate}[(i)]
  \item
  \label{item:GeneralConstructionFirst}
  For all $k=0,1,2,\ldots$ and for all $J\in\cG_k,$
  there exists $I\in\cG_{k+1}$ with $I\subseteq J.$
  \item
  \label{item:GeneralConstructionSecond}
  For all $k=0,1,2,\ldots$ and for all $I\in\cG_{k+1},$
  there exists $J \in \cG_k$ such that $I \subseteq J.$
  \item
  \label{item:GeneralConstructionVanishingLength}
  $\lim_{k\to\infty} \max\ci{I\in\cG_k} \measure{I} = 0.$
 \end{enumerate}

 At the same time, for each $k=0,1,2,\ldots,$ for every $J \in \cG_{k}$
 we will define inductively positive integers $t_{k+1}(J),$ $L_{k+1}(J)$ and $U_{k+1}(J)$
 with $L_{k+1}(J) < U_{k+1}(J).$
 They will be chosen so that, for every sequence $\{J_k\}_{k=0}^{\infty}$ of arcs with
 $J_k \in \cG_k$ and $J_{k+1} \subseteq J_k,$ $k=0,1,2,\ldots$ the following properties hold:
 \begin{enumerate}[(i)]
  \setcounter{enumi}{3}
  \item
  \label{item:GeneralConstructionIntegerSequences}
  $L_1(J_0) = 1,$ $t_1(J_0)=1,$ the sequence $\{t_{k+1}(J_k)\}^{\infty}_{k=0}$ is strictly increasing
  and
  \begin{equation*}
   U_{k+1}(J_k) = N_{t_{k+1}(J_k)},\qquad
   L_{k+2}(J_{k+1}) = M_{t_{k+1}(J_{k})+1},
  \end{equation*}
  for all $k=0,1,2,\ldots.$
  In particular $L_{k+2}(J_{k+1}) = U_{k+1}(J_k)+Q,$ for all $k=0,1,2,\ldots.$
  \item
  \label{item:GeneralConstructionSizeControl}
  For all $k=0,1,2,\ldots$ and for all $J \in \cG_k,$
  there will hold that
  \begin{equation}
   \label{eq:GeneralConstructionSizeControl}
   \eta \leq \measure{f^{U_{k+1}(J_k)}(I)} \leq 4\eta
  \end{equation}
  for all $I \in \cG_{k+1}$ with $I\subseteq J$
  and where $0 < \eta < 1/4$ is the constant given by
  Lemma~\ref{lemma:AlternativeInductiveStep}, for all $k=0,1,2,\ldots.$
  \item
  \label{item:GeneralConstructionVanishingBlocks}
  There will hold
  \begin{equation}
  \label{eq:GeneralConstructionVanishingBlocks}
   \lim_{k\to\infty} \sum_{n=U_{k+1}(J_k)}^{U_{k+2}(J_{k+1})} |a_n| = 0.
  \end{equation}
  \item
  If we set
  \begin{equation*}
   F_k^{J_k} \coloneq \sum_{n=1}^{U_{k+1}(J_k)} a_nf^n,\quad k=0,1,2,\ldots,
  \end{equation*}
  and
  \begin{equation*}
   d_{k}^{J_k} \coloneq \max\{|F_{k-1}^{J_{k-1}}(\xi)-w|\colon \xi \in J_k\},\quad k=1,2,\ldots,
  \end{equation*}
  then there will hold
  \begin{equation}
   \label{eq:GeneralConstructionVanishingErrors}
   \lim_{k\to\infty} d_k^{J_k} = 0.
  \end{equation}
 \end{enumerate}

 Let $\{J_k\}^{\infty}_{k=0}$ be a sequence of arcs as before
 and set
 \begin{equation*}
  \alpha_k^{J_k} \coloneq \max\{|F_{k-1}^{J_{k-1}}(\xi)-F_{k-1}^{J_{k-1}}(\xi')|\colon
  \xi,\xi' \in J_{k}\},\quad k=1,2,\ldots.
 \end{equation*}
 Observe that after the construction is complete, Corollary~\ref{corl:OscillationControlSequence} coupled with estimate~\eqref{eq:GeneralConstructionSizeControl} (without using~\eqref{eq:GeneralConstructionVanishingBlocks} or~\eqref{eq:GeneralConstructionVanishingErrors}) will yield
 \begin{equation}
  \label{eq:VanishingIntermediateErrors}
  \lim_{k\to\infty} \alpha_k^{J_k} = 0,
 \end{equation}
since $\lim_{n\rightarrow\infty}a_n=0$.
 Moreover, once we are done, combining~\eqref{eq:GeneralConstructionVanishingBlocks} and~\eqref{eq:GeneralConstructionVanishingErrors} we will have that
 \begin{equation*}
  \sum_{n=1}^{\infty} a_nf_n(\xi) = w,
 \end{equation*}
 where $\xi$ is the unique point of $\bigcap_{k=0}^{\infty} J_k,$
 for every sequence of arcs $\{J_k\}^{\infty}_{k=0}$ with $J_k \in \cG_k$ and $J_{k+1} \subseteq J_k$ for $k=0,1,2\ldots.$

 Finally, for all $k = 1,2,\ldots$ and for all $J\in\cG_{k-1},$
 we will define nonempty families $\{\cF_{t}(J)\}$ of pairwise disjoint subarcs of $J,$
 where $t$ ranges over $1,\ldots,t_{k}(J)$ if $k = 1,$
 and over $t_{k-1}(J')+1,\ldots,t_k(J)$ if $k > 1,$
 where in the latter case $J'$ is the unique arc in $\cG_{k-2}$ with $J \subseteq J'.$
 These nonempty families of arcs will satisfy for all $k=1,2,\ldots$ the following properties:
 \begin{enumerate}[(i)]
 \setcounter{enumi}{7}
 \item
 \label{item:GeneralConstructionIterativeLemmaFirst}
 For each admissible $t \leq t_{k}(J)-1$ and for each $I \in \cF_{t+1}(J),$
 there exists $L \in \cF_{t}(J)$ such that $I \subseteq L.$
 \item
 For each admissible $t \leq t_{k}(J)-1$ and for each $L \in \cF_{t}(J),$
 there exists $I \in \cF_{t+1}(J)$ such that $I \subseteq L.$
 \item
 For each $t \leq t_k(J)-1$ there holds
 \begin{equation*}
  \max_{I \in \cF_{t+1}(J)}
   \left\{\frac{\measure{I}}{\measure{I'}}\right\} \leq \e_{t+1}
 \end{equation*}
 where $I'$ is the unique arc in $\cF_{t}(J)$ with $I \subseteq I'.$
 \item
 For each admissible $t \leq t_k(J)-1$ and for each $L \in \cF_t(J),$
 there holds
 \begin{equation*}
  \sum_{\substack{I\in\cF_{t+1}(J)\\ I \subseteq L}} \measure{I} \geq C \cdot \measure{L}.
 \end{equation*}
 \item
 \label{item:GeneralConstructionIterativeLemmaLast}
 There holds $\cF_{t_k(J)}(J) = \{I \in \cG_{k}\colon I \subseteq J\}.$
\end{enumerate}

 To start the process, we just apply Corollary~\ref{corl:BuildingBlock}
 on $J_0$ between $M_1 = 1$ and $N_1.$
 So $\cG_0 \coloneq \{J_0\},$ $L_{1}(J_0) = M_1 = 1,$
 $U_1(J_0) = N_1,$ $t_1(J_0) = 1$ and $\cG_1 = \cF_1(J_0)$
 is the last family of subarcs of $J_0$ we obtain.
 Conditions~\eqref{item:GeneralConstructionFirst}--\eqref{item:GeneralConstructionSecond},
 \eqref{item:GeneralConstructionIterativeLemmaFirst}--\eqref{item:GeneralConstructionIterativeLemmaLast}
 and \eqref{item:GeneralConstructionIntegerSequences}--\eqref{item:GeneralConstructionSizeControl}
 are satisfied in this first step.

 Now let us describe the inductive step.
 Assume that we have defined all the various objects that are indexed by indices up to and including $k,$
 for some positive integer $k.$
 We want to define the objects that will be indexed by $k+1.$

 Let $J\in\cG_{k}$ be arbitrary.
 We want to define $L_{k+1}(J_k),U_{k+1}(J_k)$
 and $\{I \in \cG_{k+1}\colon I \subseteq J\}.$
 Let us set $J_k \coloneq J$ and let $J_0,\ldots,J_{k-1}$ be the predecessors
 of $J$ in $\cG_0,\ldots,\cG_{k-1}$ respectively.
 We have by the inductive hypothesis
 \begin{equation*}
  \eta \leq \measure{f^{U_{k}(J_{k-1})}(J)} \leq 4\eta,
 \end{equation*}
 so by part~\eqref{item:BoundaryToInterior} of Lemma~\ref{lemma:RelationInteriorBoundary}
 we have $|f^{U_k(J_{k-1})}(z(J))| \leq \gamma.$
 Pick $z^{\ast}(J)$ on the same radius as $z(J)$ so that $dI(z^{\ast}(J)) = J$
 (as it was done in the proof of Proposition~\ref{prop:LongBlocksConstruction}).
 Then, by Schwarz Lemma we have $|f^{U_{k}(J_{k-1})}(z^{\ast}(J))| \leq \gamma_1.$
 So we set $L_{k+1}(J_k) \coloneq U_{k}(J_{k-1}) + Q$ and
 observe that $|f^{L_{k+1}(J_k)}(z^{\ast}(J))| < \e.$
 Moreover, we have $U_{k}(J_{k-1}) = N_{t_{k}(J_{k-1})},$
 thus $L_{k+1}(J_k) = M_{t_{k}(J_{k-1})+1}$
 from the way the sequences $\{M_l\}^{\infty}_{l=1}$ and $\{N_l\}^{\infty}_{l=1}$ were defined.
 For brevity, we set $t \coloneq t_{k}(J_{k-1}).$

 To continue, we distinguish two cases.
 Let $r = r(f,\beta,Q,N) > 0$ be the constant given by Corollary~\ref{corl:BuildingBlock}.

 \textbf{Case 1.} There holds
 \begin{equation*}
  \max_{l\geq t} \sum_{n=N_l}^{N_{l+1}} |a_n| \leq \frac{r d_{k}^{J_{k}}}{2}.
 \end{equation*}
 We pick the smallest $t_1 \geq t+1$ such that
 \begin{equation*}
  \sum_{n=M_{t+1}}^{N_{t_1}} |a_n| \geq \frac{rd_{k}^{J_{k}}}{2}.
 \end{equation*}
 Set $t_{k+1}(J_k) \coloneq t_1.$
 The minimality of $t_1$ and the assumption for the present case gives
 \begin{equation*}
  rd_{k}^{J_{k}} \geq \sum_{n=M_{t+1}}^{N_{t_1}} |a_n|
  \geq \frac{r d_{k}^{J_{k}}}{2}.
 \end{equation*}
 We set $U_{k+1}(J_k) \coloneq N_{t_1}.$
 Pick $\xi_k \in J_k$ with
 \begin{equation*}
  |w-F_{k-1}^{J_{k-1}}(\xi_k)|=d_k^{J_k}.
 \end{equation*}
 We apply then Corollary~\ref{corl:BuildingBlock} from $t+1$ to $t_1,$
 with $w$ replaced by $w-F_{k-1}^{J_{k-1}}(\xi_k),$
 and we let $\{I \in \cG_{k+1}\colon I \subseteq J\}$ be
 the final family of subarcs of $J_k$ we obtain.
 We also let $\cF_{t_{k}(J_{k-1})+1}(J_k),\ldots,\cF_{t_{k+1}(J_k)}(J_k)$ be
 the sequence of intermediate families of arcs we obtain.

 Let $I$ be any arc of $\{I \in \cG_{k+1}\colon I \subseteq J\}.$
 Let $\xi \in I$ be arbitrary.
 Then, we have
 \begin{align*}
  \left|w-F_{k-1}^{J_{k-1}}(\xi_k)-\sum_{n=M_{t+1}}^{N_{t_1}}a_nf^n(\xi)\right|
  &\leq |w-F_{k-1}^{J_{k-1}}(\xi_k)| - r\sum_{n=M_{t+1}}^{N_{t_1}}|a_n|\\
  &\leq d_{k}^{J_{k}}\left(1-\frac{r^2}{2}\right),
 \end{align*}
 therefore
 \begin{equation*}
  |w-F_{k}^{J_k}(\xi)| \leq
  \left(1-\frac{r^2}{2}\right) d_{k}^{J_{k}} + \alpha_k^{J_k} + \sum_{n=N_{t}}^{M_{t+1}-1}|a_n|.
 \end{equation*}
 In particular, we get that
 \begin{equation*}
  d_{k+1}^{I} \leq
  \left(1-\frac{r^2}{2}\right) d_{k}^{J_{k}} + \alpha_k^{J_k}
  + \sum_{n=U_{k}(J_{k-1})}^{L_{k+1}(J_k)}|a_n|.
 \end{equation*}

 \textbf{Case 2.} There holds
 \begin{equation*}
  \max_{l\geq t}\sum_{n=N_l}^{N_{l+1}}|a_n|>\frac{r d_{k}^{J_{k}}}{2}.
 \end{equation*}
 In this case, we just pick $U_{k+1}(J_k) \coloneq N_{t+1}$ and
 set $t_{k+1}(J_k) \coloneq t+1.$
 We apply Corollary~\ref{corl:BuildingBlock} on $J_k$ between $M_{t+1}$ and $N_{t+1}.$
 We let $\{I \in \cG_{k+1} \colon I \subseteq J\} = \cF_{t+1}(J)$ be
 the last family of subarcs of $J_k$ we obtain.
 As in Case 1, pick $\xi_k \in J$ with
 \begin{equation*}
  |w-F_{k-1}^{J_{k-1}}(\xi_k)| = d_k^{J_k}.
 \end{equation*}
 Let $I$ be any arc of $\{I \in \cG_{k+1} \colon I \subseteq J\}.$
 Let $\xi \in I$ be arbitrary.
 Then, we have
 \begin{align*}
  &\left|\sum_{n=1}^{N_{t+1}}a_nf^n(\xi)-w\right| \leq
  |F_{k-1}^{J_{k-1}}(\xi)-w| + \sum_{n=N_t+1}^{N_{t+1}}|a_n|\\
  &\leq |F_{k-1}^{J_{k-1}}(\xi)-F_{k-1}^{J_{k-1}}(\xi_k)|
  + |F_{k-1}^{J_{k-1}}(\xi_k)-w| + \sum_{n=N_t+1}^{N_{t+1}}|a_n|\\
  &\leq \alpha_k^{J_{k}} + d_k^{J_k} + \sum_{n=N_t+1}^{N_{t+1}}|a_n|
  \leq \alpha_k^{J_{k}} + \left(1+\frac{2}{r}\right) \max_{l\geq t}\sum_{n=N_l}^{N_{l+1}} |a_n|.
 \end{align*}
 Therefore, in this case we have that
 \begin{equation*}
  d_{k+1}^{I} \leq \alpha_k^{J_{k}}
  + \left(1+\frac{2}{r}\right) \max_{l\geq t} \sum_{n=N_l}^{N_{l+1}} |a_n|.
 \end{equation*}

 This completes the inductive construction.
 Let us check that all required conditions are satisfied.
 Conditions~\eqref{item:GeneralConstructionFirst}--\eqref{item:GeneralConstructionVanishingLength}
 and~\eqref{item:GeneralConstructionIterativeLemmaFirst}--\eqref{item:GeneralConstructionIterativeLemmaLast}
 are clearly satisfied due to Corollary~\ref{corl:BuildingBlock}.
 Moreover, the construction and Lemma~\ref{lemma:AlternativeInductiveStep}
 also ensure conditions~\eqref{item:GeneralConstructionIntegerSequences}
 and~\eqref{item:GeneralConstructionSizeControl}.

 Next, we show that condition~\eqref{eq:GeneralConstructionVanishingErrors} is satisfied.
 Let $\{J_k\}^{\infty}_{k=0}$ be any sequence of arcs with $J_{k+1} \subseteq J_k$ and
 $J_k \in \cG_k$ for each $k \geq 0.$
 The construction yields
 \begin{equation*}
  d_{k}^{J_k} \leq
  \left(1-\frac{r^2}{2}\right) d_{k-1}^{J_{k-1}}
  + \alpha_{k-1}^{J_{k-1}} + \sum_{n=U_{k-1}(J_{k-2})}^{L_{k}(J_{k-1})} |a_n|
  + \left(1+\frac{2}{r}\right) \max_{l \geq t_{k-1}(J_{k-2})} \sum_{n=N_l}^{N_{l+1}} |a_n|
 \end{equation*}
 for all $k \geq 2.$
 Now, using that
 \begin{equation*}
  1-\frac{r^2}{2} < 1,
 \end{equation*}
 observation~\eqref{eq:VanishingIntermediateErrors},
 the fact that $L_{k}(J_{k-1}) - U_{k-1}(J_{k-2}) = Q,$
 the assumption that $\lim_{n\to\infty} |a_n| = 0$ and
 \begin{equation*}
  \lim_{t\to\infty} \max_{l \geq t} \sum_{n=N_l}^{N_{l+1}} |a_n| = 0,
 \end{equation*}
 we conclude that $\lim_{k\to\infty} d_k^{J_k} = 0.$

 Finally, we establish condition~\eqref{item:GeneralConstructionVanishingBlocks}.
 The construction shows that
 \begin{equation*}
  \sum_{n=U_{k+1}(J_k)}^{U_{k+2}(J_{k+1})} |a_n| \leq
  \max \left(\sum_{n=U_{k+1}(J_k)}^{L_{k+2}(J_{k+1})} |a_n| + rd_{k+1}^{J_{k+1}},
  \sum_{n=N_{t_{k+1}(J_k)}}^{N_{t_{k+1}(J_k)+1}} |a_n|\right)
 \end{equation*}
 for all $k \geq 0.$
 Since $\lim_{k\to\infty} d_k^{J_k} = 0$ and
 \begin{equation*}
  \lim_{t\to \infty} \sum_{n=N_t}^{N_{t+1}} |a_n| = 0
 \end{equation*}
 we conclude immediately the desired result.

 We define the Cantor like set
 \begin{equation*}
  E \coloneq \bigcap_{k=1}^{\infty} \bigcup_{I \in \cG_k} I.
 \end{equation*}
 Then it is clear that
 \begin{equation*}
  \sum_{n=1}^{\infty} a_nf^{n}(\xi) = w
 \end{equation*}
 for all $\xi \in E,$ so that $E \subseteq A(w).$
 Therefore, by monotonicity of Hausdorff measure, once we see that $\hausdorff{\varphi}{E} > 0$
 we will have proved the theorem.
 To this end, we make use of Lemma~\ref{lemma:HausdorffMeasure}.
 Set $\cF_0 \coloneq \{J_0\}$ and
 \begin{equation*}
  \cF_{t} \coloneq \{I\colon \text{ there is } k \geq 0 \text{ and } J \in \cG_k
  \text{ such that } I \in \cF_t(J) \text{ for some admissible } t\}.
 \end{equation*}
 It is clear that $\{\cF_t\}^{\infty}_{t=0}$ satisfies
 the conditions of Lemma~\ref{lemma:HausdorffMeasure} by the previous construction
 and that
 \begin{equation*}
  E = \bigcap_{t=0}^{\infty} \bigcup_{I \in \cF_t} I,
 \end{equation*}
 thus concluding the proof.
\end{proof}

 \section{Theorem~\ref{thm:PaleyWeissDimensionThm} is optimal}
 \label{sec:Optimality}
 The main goal of this section is to prove the Theorem~\ref{thm:Optimality},
 which we restate here for the reader's convenience.
 \theoremstyle{plain}
 \newtheorem*{theorem:Optimality}{Theorem~\ref{thm:Optimality}}
 \begin{theorem:Optimality}
  Consider a measure function $\varphi$ such that
  \begin{equation}
   \label{eq:MeasureFunctionRestrictiveness}
   \lim_{t \to 0^+} \frac{\varphi(t)}{t^s} = 0
  \end{equation}
  for any $0 < s < 1.$
  Then, there exist a finite Blaschke product $f$ and
  a sequence $\{a_n\} = \{a_n\}(\varphi,f)$ of complex numbers tending to zero with
  $\sum_n |a_n| = \infty$ such that for any $w \in \complex,$ the set
  \begin{equation*}
   A(w) \coloneq \left\{\xi \in \ddisk\colon \sum_{n=1}^\infty a_n f^n(\xi)
   \text{ converges and } \sum_{n=1}^\infty a_n f^n(\xi) = w\right\}
  \end{equation*}
  has Hausdorff measure $\hausdorff{\varphi}{A(w)} = 0.$
 \end{theorem:Optimality}
 In fact, one can see that our construction yields something a bit stronger.
 Not only we will have $\sum_n |a_n| = \infty,$ but also $\sum_n |a_n|^2 = \infty.$
 This turns out to be absolutely necessary for our argument.
 For this reason, one could ask if it is possible to construct a similar example with a sequence
 with $\sum_n |a_n|^2 < \infty$ or if one can have a stronger conclusion
 when considering only such sequences.

 First, we can assume that
 \begin{equation}
  \label{eq:MeasureFunctionLessThanLebesgue}
  \lim_{t \to 0^+} \frac{t}{\varphi(t)} = 0
 \end{equation}
 because for any $w \in \complex$ we must have $\measure{A(w)} = \hausdorff{1}{A(w)} = 0$
 (otherwise, $\sum_n a_n f^n$ would be a constant function).
 Let us consider, for the rest of this section, a fixed measure function $\varphi$
 satisfying~\eqref{eq:MeasureFunctionRestrictiveness} and
 also~\eqref{eq:MeasureFunctionLessThanLebesgue}.
 In other words, we consider a fixed measure function that is more restrictive
 than any power function $\varphi_s(t) = t^s$ for any $0 < s <1,$
 but with the Lebesgue measure still being more restrictive than $\varphi.$
 Define
 \begin{equation*}
  \psi(t) \coloneq \frac{\varphi(t)}{t},
 \end{equation*}
 which we will assume without loss of generality to be strictly decreasing.
 Moreover, by~\eqref{eq:MeasureFunctionLessThanLebesgue} we have that
 $\lim_{t \to 0^+} \psi(t) = \infty.$
 For our later convenience, observe that
 condition~\eqref{eq:MeasureFunctionRestrictiveness} can be expressed in the following way.
 For any $0 < s < 1$ there exists $t_s > 0$ such that
 \begin{equation*}
  \psi(t) \leq t^{-s},\qquad \text{ for all } t \leq t_s.
 \end{equation*}
 Again without loss of generality, we can also assume that,
 if $s_1 > s_2,$ then $t_{s_1} > t_{s_2}.$

 To construct the examples that show Theorem~\ref{thm:Optimality},
 we will use finite Blaschke products of the form $f(z) = z^\nu,$
 for some integer $\nu \geq 2.$
 Thus, consider $\nu \geq 2$ to be fixed from now on.
 For a given sequence $\{a_n\}$ of complex numbers, we will denote
 \begin{equation*}
  F(z) \coloneq \sum_{n=1}^\infty a_n f^n(z) = \sum_{n=1}^\infty a_n z^{\nu^n},\qquad
  z \in \disk,
 \end{equation*}
 and we will refer to its partial sums as
 \begin{equation*}
  F_N(z) \coloneq \sum_{n=1}^{N} a_n f^n(z) = \sum_{n=1}^N a_n z^{\nu^n},\qquad
  z \in \disk,\, N \geq 1.
 \end{equation*}
 For completeness, we define $F_0 \equiv 0.$ 
 Our construction of a suitable sequence $\{a_n\}$ is based on
 an argument due to Makarov~\cite[Section~5.C]{ref:MakarovConformalMappings}.

 Note that $F$ is a function in the Bloch space
 (see~\cite{ref:AndersonCluniePommerenkeBlochFunctions} for the result for lacunary series
 and~\cite[Theorem~4.5]{ref:NicolauConvergenceIterates} for iterates of general inner functions).
 In particular, the averages
 \begin{equation*}
  \frac{1}{\measure{I}} \int_I F(\xi)\, \mathdm(\xi)
 \end{equation*}
 are well defined for any arc $I \subseteq \ddisk$
 (see~\cite[Section~2]{ref:LlorenteMartingalesAndApplications}
 for an exposition of this and other facts relating Bloch functions and dyadic martingales).
 We will use these averages to define a martingale as follows.
 Consider, for $n \geq 0,$ the collection $\cF_n$ of $\nu$-adic arcs
 of length $\nu^{-(n+1)}.$
 In other words, we set
 \begin{equation*}
  \cF_n \coloneq \{[\exp(2 \pi i k \nu^{-(n+1)}),\exp(2 \pi i (k+1) \nu^{-(n+1)}))
  \colon 0 \leq k < \nu^{n+1}\}
 \end{equation*}
 for $n \geq 0.$
 We define the $\nu$-adic martingale $\{M_n\}_{n=0}^\infty$ as
 \begin{equation}
  \label{eq:MartingaleDefinition}
  M_n(\xi) \coloneq (F|\cF_n)(\xi)
  = \frac{1}{\measure{I(\xi)}} \int_{I(\xi)} F(z)\, \mathdm(z),
 \end{equation}
 where $I(\xi)$ is the unique arc in $\cF_n$ with $\xi \in I(\xi).$
 The following lemma states that the martingale $\{M_n\},$
 at a given time $N,$ only sees the partial sum $F_N.$

 \begin{lemma}
  \label{lemma:MartingalePartialSums}
  Consider $F,$ $F_n$ and $\{M_n\}$ defined as before.
  Then, for a given integer $N \geq 0,$ we have that
  \begin{equation*}
   M_N(\xi) = (F_N|\cF_N)(\xi)
   = \frac{1}{\measure{I(\xi)}} \int_{I(\xi)} F_N(z)\, \mathdm(z).
  \end{equation*}
 \end{lemma}
 \begin{proof}
  Fix a positive integer $m > N$ and consider $I \in \cF_N.$
  Observe that by the change of variable formula it holds that
  \begin{equation*}
   \frac{1}{\measure{I}} \int_I z^{\nu^m}\, \mathdm(z)
   = \frac{1}{\nu^{N+1}\measure{I}} \int_{\ddisk} z^{\nu^{m-N-1}}\, \mathdm(z).
  \end{equation*}
  Since $m > N,$ the last integral vanishes.

  Next, write
  \begin{equation*}
      F(z) = F_N(z) + F_{>N}(z),
  \end{equation*}
  so that by linearity we get
  \begin{equation*}
      M_N(\xi) = (F_N|\cF_N)(\xi) + (F_{>N}|\cF_N)(\xi).
  \end{equation*}
  Using the same change of variable as before to the function $F_{>N},$ we see that
  \begin{equation*}
      (F_{>N}|\cF_N)(\xi) =
      \frac{1}{\nu^{N+1}\measure{I(\xi)}}
      \int_{\ddisk} \sum_{m>N} a_{m} z^{\nu^{m-N-1}}\, \mathdm(z)
  \end{equation*}
  and the function in the last integral is an analytic function on the disk
  whose value at the origin is zero.
  Therefore, by Cauchy's formula we get that
  \begin{equation*}
      (F_{>N}|\cF_N)(\xi) = 0,
  \end{equation*}
  which concludes the proof.
 \end{proof}

 For a given integer $n \geq 1,$ we define the increment $\Delta M_n$ of
 the martingale $\{M_n\}$ by
 \begin{equation*}
  \Delta M_n(\xi) \coloneq M_n(\xi) - M_{n-1}(\xi).
 \end{equation*}
 Since $\{M_n\}$ is a martingale,
 the increments $\{\Delta M_n\}$ are orthogonal in the $\Lp{2}$ sense.
 Later in our construction, we will set some of the coefficients $a_n$ equal to zero.
 Next result gives an estimate on $|\Delta M_n|$ when we have various null coefficients.

 \begin{lemma}
  \label{lemma:IncrementsWithNullCoefficients}
  Consider $F,$ $F_n,$ $\{M_n\}$ and the increments $\Delta M_n$ defined as before.
  Fix integers $N \geq 1$ and $0 \leq m \leq N-1.$
  In addition, assume that $\sup_n |a_n| \leq 1$ and that $a_n = 0$ for $N-m \leq n \leq N-1$
  (so if $m = 0,$ this last condition applies to no coefficients).
  Then there exist constants $C > 0$ and $0 < c = c(\nu) < 1$ such that
  \begin{equation*}
   \big||\Delta M_N| - c|a_N|\big| \leq C \nu^{-m}.
  \end{equation*}
 \end{lemma}
 \begin{proof}
  First, consider a positive integer $k \leq N.$
  Let $I$ be any arc of $\cF_N.$
  Then, by the change of variable formula, it holds that
  \begin{equation*}
   \frac{1}{\measure{I}} \int_I z^{\nu^k}\, \mathdm(z)
   = \frac{1}{\nu^k\measure{I}} \int_J z\, \mathdm(z)
   = \frac{1}{\measure{J}} \int_J z\, \mathdm(z),
  \end{equation*}
  where $J$ is the arc in $\cF_{N-k}$ which is the image of $I$ under $z^{\nu^k}.$
  Therefore, for a given $\xi \in \ddisk,$
  we can express the increment $\Delta M_N (\xi)$ as
  \begin{equation}
   \label{eq:IncrementExpression}
   \begin{split}
    \Delta &M_N(\xi) = a_N \frac{1}{\measure{I_0}} \int_{I_0} z\, \mathdm(z)\\
    &+ \sum_{k=1}^{N-1} a_k
    \left(\frac{1}{\measure{I_{N-k}}} \int_{I_{N-k}} z\, \mathdm(z)
    - \frac{1}{\measure{I_{N-k-1}}} \int_{I_{N-k-1}} z\, \mathdm(z)\right),
   \end{split}
  \end{equation}
  where for $0 \leq j \leq N-1$ we denote by $I_j$ the arc in $\cF_j$ that contains
  $\xi^{\nu^{N-j}}$ or, in other words,
  the image of the arc $I \in \cF_N$ that contains $\xi$ by the mapping $z^{\nu^{N-j}}.$
  In particular, observe that for $1 \leq j \leq N-1$ the arc $I_{j-1}$ is
  the unique arc in $\cF_{j-1}$ that contains $I_j.$

  Next, we estimate the differences in the last sum of~\eqref{eq:IncrementExpression}.
  Take an arc $I \in \cF_j$ with $1 \leq j \leq N-1$ and
  denote by $J$ the unique arc in $\cF_{j-1}$ such that $I \subseteq J.$
  We start by expressing
  \begin{equation}
   \label{eq:DifferentGenerationsAverageExpression}
   \begin{split}
    \frac{1}{\measure{I}} &\int_{I} z\, \mathdm(z)
    - \frac{1}{\measure{J}} \int_{J} z\, \mathdm(z)\\
    &= \frac{1}{\measure{J}}
    \sum_{l=1}^{\nu} \left(\int_I z\, \mathdm(z) - \int_{I_l} z\, \mathdm(z)\right),
   \end{split}
  \end{equation}
  where $I_l,$ $1 \leq l \leq \nu,$ are all the arcs in $\cF_j$ contained in $J.$
  Since all arcs in consideration have length less than $1,$
  for each such arc $K$ we can define its anticlockwise extreme point $z_+(K)$ and
  its clockwise extreme point $z_-(K).$
  Thus, for such an arc $K$ it happens that
  \begin{equation*}
   \int_K z\, \mathdm(z) = \frac{-i}{2\pi} (z_+(K)-z_-(K)).
  \end{equation*}
  Let us denote by $z_T$ the complex number parallel to the tangent to $\ddisk$
  at point $z_-(J)$ (oriented in the anticlockwise sense) with $|z_T| = \measure{I}.$
  Then, if we enumerate the arcs $I_l$ consecutively starting from $z_-(J),$
  elementary trigonometry shows that
  \begin{equation}
   \label{eq:RightAngleApprox}
   \begin{split}
    |z_+(I_l) - (z_-(J)+lz_T)| \leq C (\measure{I})^2,\\
    |z_-(I_l) - (z_-(J)+(l-1)z_T)| \leq C (\measure{I})^2,
   \end{split}
  \end{equation}
  for some constant $C > 0$ independent of $j$ and $I$ and for each $1 \leq  l \leq \nu.$
  Hence, for each $1 \leq l \leq \nu,$ we find that
  \begin{equation}
   \label{eq:SameGenerationAverageComparison}
   \left|\int_I z\, \mathdm(z) - \int_{I_l} z\, \mathdm(z)\right|
   \leq |z_+(I) - z_-(I) - z_+(I_l) + z_-(I_l)| \leq C \cdot \measure{I}^2,
  \end{equation}
  where in the last inequality we have used~\eqref{eq:RightAngleApprox} and
  where $C > 0$ denotes an absolute positive constant possibly different from
  the one appearing in~\eqref{eq:RightAngleApprox}.
  Joining~\eqref{eq:DifferentGenerationsAverageExpression}
  and~\eqref{eq:SameGenerationAverageComparison} we get that
  \begin{equation}
   \label{eq:AverageDifferencesDecay}
   \left|\frac{1}{\measure{I}} \int_{I} z\, \mathdm(z)
       - \frac{1}{\measure{J}} \int_{J} z\, \mathdm(z)\right|
   \leq C \nu \frac{(\measure{I})^2}{\measure{J}}
   = C \nu^{-j}.
  \end{equation}

  Now, take
  \begin{equation*}
   c = \left|\frac{1}{\measure{I}} \int_{I} z\, \mathdm(z)\right|,
  \end{equation*}
  where $I$ is any arc in $\cF_0$
  (it is easy to see that the result does not depend on the choice of $I,$
  only on its length and, hence, on the choice of $\nu$).
  Then, expression~\eqref{eq:IncrementExpression} gives that
  \begin{equation*}
   \begin{split}
    \big||\Delta M_N| &- c|a_N|\big|
    \leq \left|\Delta M_N - a_N \frac{1}{\measure{I_0}} \int_{I_0} z\, \mathdm(z)\right|\\
    &\leq \sum_{k=1}^{N-1} |a_k|
    \left|\frac{1}{\measure{I_{N-k}}} \int_{I_{N-k}} z\, \mathdm(z)
    - \frac{1}{\measure{I_{N-k-1}}} \int_{I_{N-k-1}} z\, \mathdm(z)\right|.
   \end{split}
  \end{equation*}
  So, using the assumptions that $|a_n| \leq 1$ and $a_k = 0$ for $N-m \leq k \leq N-1$
  together with estimate~\eqref{eq:AverageDifferencesDecay}, we get that
  \begin{equation*}
   \big||\Delta M_N| - c|a_N|\big| \leq C \sum_{j=m+1}^\infty \nu^{-j}
   \leq C \nu^{-m},
  \end{equation*}
  where the last instance of $C$ denotes a possibly different positive constant,
  although still an absolute one, as we wanted to see.
 \end{proof}

 Let us denote the $\Lp{2}$ norm of $M_N$ by
 \begin{equation*}
  \sigma(N) \coloneq \left(\int_{\ddisk} |M_N(\xi)|^2\, \mathdm(\xi)\right)^{1/2}.
 \end{equation*}
 Recall also that, since the increments of the martingale are $\Lp{2}$ orthogonal and
 in our construction $M_0 = 0,$ we have that
 \begin{equation*}
  \sigma(N) = \left(\sum_{k=1}^N \int_{\ddisk} |\Delta M_k(\xi)|^2\, \mathdm(\xi)\right)^{1/2}.
 \end{equation*}
 In the proof of Theorem~\ref{thm:Optimality}, we will construct sequences $\{a_n\}$
 in which the nonzero coefficients will be separated by several null ones.
 The amount of null coefficients between two nonzero ones will be chosen so that,
 if $|a_n| > 0,$ then by Lemma~\ref{lemma:IncrementsWithNullCoefficients}
 we have $|\Delta M_n(\xi)| \simeq |a_n|,$
 with the comparability constants independent of $n.$
 Furthermore, the nonzero coefficients will be chosen to be such that
 $\sigma(N) \longrightarrow \infty$ as $N \to \infty.$
 The reason for this choice is the following lemma about $\Lp{2}$ divergent martingales
 that resemble a random walk,
 which relates the probability of the martingale to remain inside a certain boundary
 with the divergence of $\sigma(N)^2.$
 Here, for an increasing (possibly finite) sequence of positive integers $\tau = \{n_l\},$
 denote
 \begin{equation*}
    \overline{\sigma}(N,\tau) = \overline{\sigma}(N,\{n_l\}) = \left(\sum_{n_l\colon n_l \leq N}
                                \int_{\ddisk} |\Delta M_{n_l}(\xi)|^2\, \mathdm(\xi)\right)^{1/2}.
 \end{equation*}
 Also for later convenience,
 given a sequence $\tau = \{n_l\}$ of positive integers as before and a positive integer $k,$
 define
 \begin{equation*}
  \overline{N}(k,\tau) \coloneq \inf \{N \geq 1\colon
                                \overline{\sigma}(N,\tau)^2 \geq k\},
 \end{equation*}
 with the convention that $\inf \emptyset = \infty.$
 Whenever the sequence $\tau$ is clear from the context, we will omit it
 and just denote $\overline{\sigma}(N)$ and $\overline{N}(k).$

 \begin{lemma}
  \label{lemma:BoundaryProbabilityDecay}
  Consider a martingale $\{M_n\}$ which is $\Lp{2}$ divergent
  with $|\Delta M_n|^2 \leq K$ for some $K > 0.$
  Moreover, assume that there is an increasing sequence $\tau = \{n_l\}$ of positive integers
  and a sequence $\{a_l\}$ of complex numbers with $\sum_l |a_l|^2 = \infty$
  and such that $C_0 |\Delta M_{n_l}(\xi)|^2 \geq |a_l|^2$ for $l \geq 1,$
  with the comparability constant $C_0 \geq 1$ independent of $l.$
  For each $R > 0,$ there exist constants $C = C(C_0,R,K) > 0$ and $0 < c = c(R,K) < 1$ so that
  \begin{equation*}
   \measure{\set{\xi \in \ddisk\colon \max_{1 \leq n \leq \overline{N}(k,\tau)} |M_n(\xi)| \leq R}}
   \leq C c^k
  \end{equation*}
  for all integer $k \geq 1.$
 \end{lemma}
 \begin{proof}
  Let us denote
  \begin{equation*}
   A(R,N) = \set{\xi \in \ddisk\colon \max_{1 \leq n \leq N} |M_n(\xi)| \leq R}
  \end{equation*}
  and observe that these sets are decreasing with $N.$
  We will use the maximal inequality
  \begin{equation}
   \label{eq:KolmogorovReverseInequality}
   \measure{\set{\xi \in \ddisk\colon \max_{1 \leq n \leq N} |M_n(\xi)| \leq R}}
   \leq \frac{(R+K)^2}{\sigma(N)^2},
  \end{equation}
  which can be deduced by basic properties of dyadic martingales and stopping times
  (see for instance~\cite[Section~4.4]{ref:DurrettProbabilityTheory}).
  
  First, choose an integer $m = m(R,K) > K$ such that
  \begin{equation*}
   \frac{C_0(2R+K)^2}{m-K} \leq \delta < 1.
  \end{equation*}
  Clearly $\sigma(\overline{N}(m,\tau))^2 \geq \overline{\sigma}(\overline{N}(m,\tau),\tau)^2 \geq m,$
  so by~\eqref{eq:KolmogorovReverseInequality} we have that
  \begin{equation*}
   \measure{A(R,\overline{N}(m,\tau))} \leq \frac{(R+K)^2}{\sigma(\overline{N}(m,\tau))^2} \leq \delta < 1.
  \end{equation*}
  We use the notation $\overline{N}(j)$ as the sequence $\tau$ is fixed.
  Now, assume that we have shown that
  \begin{equation*}
   \measure{A(R,\overline{N}(m(k-1)))} \leq \delta^{k-1}
  \end{equation*}
  for some integer $k > 1.$
  Note that we have $\overline{N}(m(k-1)) < \overline{N}(mk)$ because $m > K.$
  Thus, it holds that
  \begin{equation*}
   \begin{split}
    A(R,&\overline{N}(mk)) \subseteq A(R,\overline{N}(m(k-1)))\\
    &\cap
    \set{\xi \in \ddisk\colon
        \max_{\overline{N}(m(k-1)) \leq n \leq \overline{N}(mk)} |(M_n-M_{\overline{N}(m(k-1))})(\xi)| \leq 2R}.
   \end{split}
  \end{equation*}
  Next, we show that the measure of the intersection in the right-hand side
  of the previous expression
  is bounded by
  \begin{equation}
      \label{eq:IntersectionBound}
      C_0\, \measure{A(R,\overline{N}(m(k-1)))} \frac{(2R+K)^2}{m-K}.
  \end{equation}
  Observe that the set $A(R,\overline{N}(m(k-1)))$
  is a possibly empty collection of arcs of $\cF_{\overline{N}(m(k-1))}.$
  If it is empty, there is nothing to prove, so we assume that at least it contains an arc.
  Consider $I$ to be one of these arcs.
  Also consider the martingale $\{X_n\}_{n=\overline{N}(m(k-1))}^\infty$ defined by
  $X_n (\xi) = M_n(\xi)-M_{\overline{N}(m(k-1))}(\xi)$ restricted to $I.$
  Because of the $\Lp{2}$ orthogonality of martingale increments
  and the hypothesis on $|\Delta M_n|$ for $n \in \tau,$ we can see that
  \begin{equation*}
    \begin{split}
      \|X_{\overline{N}(mk)}\|_{2}^2
      &\geq C_0^{-1} (\overline{\sigma}(\overline{N}(mk))^2-\overline{\sigma}(\overline{N}(m(k-1)))^2)\\
      &\geq C_0^{-1} (m-K)
    \end{split}
  \end{equation*}
  where the $\Lp{2}$ norm is computed with respect to the normalised Lebesgue measure on $I.$
  Now, by applying the maximal inequality~\eqref{eq:KolmogorovReverseInequality}
  to the martingale $\{X_n\}$ and using the previous bound we find that
  \begin{equation*}
   \begin{split}
    \measure{\set{\xi \in I\colon
                  \max_{\overline{N}(m(k-1)) \leq n \leq \overline{N}(mk)} |(M_n-M_{\overline{N}(m(k-1))})(\xi)| \leq 2R}}\\
    \leq C_0 \frac{(2R+K)^2}{m-K} \measure{I}.
   \end{split}
  \end{equation*}
  Hence, if we sum over the arcs of $A(R,\overline{N}(m(k-1)))$ we obtain~\eqref{eq:IntersectionBound}.
  Therefore, we find out that
  \begin{equation*}
    \measure{A(R,\overline{N}(mk))}
    \leq C_0\, \measure{A(R,\overline{N}(m(k-1)))} \frac{(2R+K)^2}{m-K}
    \leq \delta^{k}.
  \end{equation*}
  Finally, using that the sets $A(R,N)$ decrease on $N,$
  if we choose $C = \delta^{-m}$ and $c = \delta^{1/m},$ we get that
  \begin{equation*}
   \measure{A(R,\overline{N}(k))} \leq C c^{k}.
  \end{equation*}
 \end{proof}

 \begin{proof}[Proof of Theorem~\ref{thm:Optimality}]
  For a fixed measure function $\varphi$ and an integer $\nu \geq 2,$
  we will show how to construct sequences $\{a_n\}$ of complex numbers tending to zero,
  which are not summable and such that $\hausdorff{\varphi}{A(w)} = 0$
  for any $w \in \complex.$
  First, observe that
  \begin{equation*}
   A(w) \subseteq \set{\xi \in \ddisk\colon \limsup_{N \to \infty} |F_N(\xi)| < \infty}.
  \end{equation*}
  Therefore, once our sequence is constructed, if we show that
  \begin{equation*}
   B(R) \coloneq \set{\xi \in \ddisk\colon \limsup_{N \to \infty} |F_N(\xi)| < R}
  \end{equation*}
  has $\hausdorff{\varphi}{B(R)} = 0$ for all $R > 0,$
  we will have the desired result.
  By the properties of the martingale $\{M_n\}$ defined by~\eqref{eq:MartingaleDefinition},
  this is equivalent to see that the sets
  \begin{equation*}
   A(R) \coloneq \set{\xi \in \ddisk\colon \sup_{n \geq 0} |M_n(\xi)| \leq R}
  \end{equation*}
  satisfy
  \begin{equation*}
   \hausdorff{\varphi}{A(R)} = 0
  \end{equation*}
  for all $R > 0.$

  Let $C > 0$ and $0 < c = c(\nu) < 1$ be the constants in Lemma~\ref{lemma:IncrementsWithNullCoefficients}
  and take an integer $m_0$ such that
  \begin{equation*}
   c_0 \coloneq c - \frac{C}{\nu^{m_0-1}} > 0.
  \end{equation*}
  Also define the sequence $\{t_l\}_{l=0}^\infty$ as follows.
  Set $t_0 = 1$ and for each $l \geq 1$ choose $0 < t_l < t_{l-1}$ so that
  \begin{equation*}
   \psi(t) \leq t^{-\nu^{-3(l+1)}}
  \end{equation*}
  for all $t \leq t_l.$

  We are ready to construct a sequence $\{a_n\}$ so that $\hausdorff{\varphi}{A(w)} = 0$
  for any $w \in \complex.$
  At the same time we will construct an increasing sequence $\tau$ of positive integers
  which will contain those indices $n$ for which $a_n \neq 0.$
  Consider $\widetilde{N}_1$ the first integer such that
  \begin{equation*}
   \nu^{-\widetilde{N}_1} \leq t_1
  \end{equation*}
  and let $\tau$ be the empty sequence.
  Now, we will choose $a_n$ such that
  \begin{equation}
   \label{eq:FirstSizeCondition}
   \begin{cases}
     1 \geq |a_n| \geq \nu^{-1} &\text{ if } n \equiv 1 \pmod{m_0+1}\\
     a_n = 0  &\text{ otherwise }
   \end{cases},
  \end{equation}
  for $1 \leq n \leq N_1$ to be determined according to the following procedure.
  Start fixing $a_1$ satisfying~\eqref{eq:FirstSizeCondition}
  and append $1$ to $\tau$ (so at the moment $\tau = (1)$).
  If $|a_1| = 1$ and $\widetilde{N}_1 = 1,$
  we set $N_1' = 1$ and proceed to fix $N_1$ as explained below.
  Otherwise, we continue fixing coefficients according to rule~\eqref{eq:FirstSizeCondition} and,
  whenever we are choosing a nonzero coefficient, we append its index to the sequence $\tau.$
  We repeat this process until we have fixed $N_1'$ coefficients,
  where $N_1'$ is the first integer $N_1' \geq \widetilde{N}_1$ such that
  \begin{equation*}
   \sum_{n=1}^{N_1'} |a_n| \geq 1.
  \end{equation*}
  Next, we keep fixing coefficients according to rule~\eqref{eq:FirstSizeCondition}
  and appending the indices of nonzero coefficients to the sequence $\tau$
  until we have chosen a total of $N_1$ of them, where $N_1$ is the first integer $N_1 \geq N_1'$ such that
  \begin{equation*}
      \overline{\sigma}(N_1,\tau)^2 \geq \left\lceil \overline{\sigma}(N_1',\tau)^2 \right\rceil.
  \end{equation*}
  Now, denote $k_1 = \lfloor\overline{\sigma}(N_1)^2\rfloor,$
  so that $\overline{N}(k_1) = N_1$ because of the minimality of $N_1$ and the fact that $|a_n| \leq 1$
  (we also set $k_0 = 0$ for completeness).

  Observe that, except for $n = 1,$ if $n \in \tau$ then $a_k = 0$ for $n-m_0 \leq k \leq n-1.$
  Thus, by Lemma~\ref{lemma:IncrementsWithNullCoefficients} we have that
  \begin{equation}
   \label{eq:FirstIncrementsEstimate}
   |\Delta M_n| \geq c|a_n| - C\nu^{-m_0} \geq \left(c-\frac{C}{\nu^{m_0-1}}\right) |a_n|
   = c_0 |a_n|
  \end{equation}
  for $n \in \tau$ with $1 \leq n \leq N_1.$
  Note as well that, since $|a_n| \geq \nu^{-1}$ for $n \equiv 1 \pmod{m_0+1}$ and because of
  estimate~\eqref{eq:FirstIncrementsEstimate},
  for each $1 \leq N \leq N_1$ we have that
  \begin{equation*}
   \overline{\sigma}(N,\tau)^2 \geq \frac{c_0^2 \nu^{-2}}{m_0+2} N,
  \end{equation*}
  which implies
  \begin{equation*}
   k \geq \frac{c_0^2 \nu^{-2}}{2(m_0+2)} \overline{N}(k,\tau)
  \end{equation*}
  for each $1 \leq k \leq k_1.$
  
  Assume now that we have already determined $k_{l-1}$ and $N_{l-1}$
  and that we have fixed $a_n$ for $n \leq N_{l-1}$ for some integer $l \geq 2.$
  Also assume that $\tau$ has been set to be the increasing sequence
  of positive integers $n \leq N_{l-1}$ for which $a_n \neq 0.$
  Lastly suppose that with all these parameters fixed, it holds that
  \begin{equation}
   \label{eq:RangeToBoundPsi}
   \nu^{-N_{l-1}} \leq t_{l-1}
  \end{equation}
  and $\overline{N}(k_{l-1},\tau) = N_{l-1}.$
  If it holds that
  \begin{equation*}
   \nu^{-N_{l-1}} \leq t_l,
  \end{equation*}
  fix $N_l = N_{l-1},$ $k_l = k_{l-1}$ and continue to the next step
  without fixing any of the coefficients $a_n.$
  Otherwise, if it is the case that $\nu^{-N_{l-1}} > t_l,$
  set $\widetilde{N}_l > N_{l-1}$ to be the first integer such that
  \begin{equation*}
   \nu^{-\widetilde{N}_{l}} \leq t_l.
  \end{equation*}
  Next, we choose $a_n$ for $N_{l-1} < n \leq N_l,$ where $N_l$ is to be determined, so that
  \begin{equation}
   \label{eq:CoefficientSizeEstimate}
   \begin{cases}
     \nu^{-l+1} \geq |a_n| \geq \nu^{-l} &\text{ if } n-N_{l-1} \equiv 0 \pmod{m_0+l}\\
     a_n = 0 &\text{ otherwise }
   \end{cases}
  \end{equation}
  according to the following procedure.
  First, we choose $a_{N_{l-1}+1}$ satisfying~\eqref{eq:CoefficientSizeEstimate}.
  If $\widetilde{N}_l = N_{l-1}+1$ and
  \begin{equation}
   \label{eq:SquareSumEstimate}
   \sum_{n=1}^{\widetilde{N}_l} |a_n| \geq l,
  \end{equation}
  we fix $N_l' = N_{l-1}+1$ and proceed to fix $N_l$ as explained below.
  Otherwise, we choose the next coefficient according to rule~\eqref{eq:CoefficientSizeEstimate}.
  In addition, every time we fix a nonzero coefficient, we append its index to the sequence $\tau.$
  We repeat this process until we have fixed up to coefficient $a_{N_l'},$
  with $N_l'$ being the first integer $N_l' \geq \widetilde{N}_l$ such that
  \begin{equation*}
   \sum_{n=1}^{N_l'} |a_n| \geq l.
  \end{equation*}
  Then we continue fixing coefficients according to rule~\eqref{eq:CoefficientSizeEstimate}
  and appending the indices of the nonzero ones to the sequence $\tau$
  until we have fixed a total of $N_l,$
  where $N_l$ is the first integer $N_l \geq N_l'$ such that
  \begin{equation}
      \label{eq:VarianceIncrement}
      \overline{\sigma}(N_l,\tau)^2 \geq \left\lceil \overline{\sigma}(N_l',\tau)^2 \right\rceil
  \end{equation}
  and denote $k_l = \lfloor\overline{\sigma}(N_l)^2\rfloor,$ so that $\overline{N}(k_l) = N_l.$
  
  Observe that, for $N_{l-1} < n \leq N_l,$ if we have $|a_n| > 0,$
  then $a_k = 0$ for $n-(m_0+l-1) \leq k \leq n-1.$
  Therefore, by Lemma~\ref{lemma:IncrementsWithNullCoefficients} we have
  \begin{equation}
   \label{eq:IncrementSizeEstimate}
   |\Delta M_n| \geq c|a_n| - C\nu^{-m_0-l+1}
   \geq \left(c-\frac{C}{\nu^{m_0-1}}\right) |a_n| = c_0 |a_n|
  \end{equation}
  for $n \in \tau$ with $N_{l-1} < n \leq N_l.$
  Take into account that, because of~\eqref{eq:CoefficientSizeEstimate}
  and~\eqref{eq:IncrementSizeEstimate},
  for each $N_{l-1} < N \leq N_l$ we have
  \begin{equation*}
   \overline{\sigma}(N,\tau)^2 \geq \frac{c_0^2 \nu^{-2l}}{m_0+l+1} N,
  \end{equation*}
  which in turn implies that for each $k_{l-1} < k \leq k_l$ it holds that
  \begin{equation}
   \label{eq:NToVarianceEstimate}
   k \geq \frac{c_0^2 \nu^{-2l}}{2(m_0+l+1)} \overline{N}(k,\tau).
  \end{equation}

  Notice that this procedure defines a sequence $\{a_n\}$ that satisfies
  the hypotheses of Theorem~\ref{thm:Optimality}.
  Indeed, since $\lim_{t \to 0^+} \psi(t) = \infty$ but $\lim_{t \to 0^+} \psi(t)/t^{-s} = 0$
  for every $s > 0,$ we have that the sequence $\{t_l\}$ tends to zero as $l \to \infty.$
  Therefore, by~\eqref{eq:RangeToBoundPsi}
  the sequence $\{N_l\}$ tends to infinity,
  which implies that at some point we define $a_n$ for every integer $n \geq 1.$
  Moreover, condition~\eqref{eq:CoefficientSizeEstimate} implies that
  $|a_n|$ tends to zero as $n \to \infty.$
  Furthermore, requirement~\eqref{eq:SquareSumEstimate} implies that
  $\sum_n |a_n| = \infty.$
  In addition, take into account that it also happens that
  $\sum_n |a_n|^2 = \infty$ because of requirement~\eqref{eq:VarianceIncrement}
  during the construction, even though this is not necessary
  for Theorem~\ref{thm:Optimality}.
  In particular, the sequence $\{k_l\}$ also tends to infinity.
  Lastly, this procedure also yields the infinite increasing sequence $\tau$
  of indices $n$ for which $a_n \neq 0.$

  Now that we have constructed a suitable sequence $\{a_n\},$
  we show that for this particular choice it happens that
  $\hausdorff{\varphi}{A(R)} = 0.$
  It is clear that $A(R) \subseteq A(R,N)$ for every positive integer $N,$ where
  \begin{equation*}
   A(R,N) \coloneq \set{\xi \in \ddisk\colon \max_{1 \leq n \leq N} |M_n(\xi)| \leq R}.
  \end{equation*}
  Observe that, because $|a_n| \leq 1$ for all $n \geq 1,$
  there is $K > 0$ such that $|\Delta M_n|^2 \leq K$ for all $n \geq 1$
  (for example, one can use Lemma~\ref{lemma:IncrementsWithNullCoefficients} to see this).
  Also, estimate~\eqref{eq:IncrementSizeEstimate} shows that there is $C_0 \geq 1$
  such that $C_0 |\Delta M_n(\xi)|^2 \geq |a_n|^2$ whenever $n \in \tau.$
  Hence, because of Lemma~\ref{lemma:BoundaryProbabilityDecay},
  there exist constants $C = C(C_0,R,K) > 0$ and $0 < c = c(R,K) < 1$ such that
  \begin{equation*}
   \measure{A(R,\overline{N}(k))} \leq C c^k
  \end{equation*}
  for every integer $k \geq 1.$
  Therefore, for a given positive integer $k$ we can cover the set
  $A(R,\overline{N}(k)) \supseteq A(R)$ using at most $C' c^k \nu^{\overline{N}(k)}$ arcs
  of length $\nu^{-\overline{N}(k)},$ for a fixed positive constant $C'.$
  Thus, we get that
  \begin{equation}
   \label{eq:HausdorffEstimate}
   \hausdorff{\varphi}{A(R)} \leq C' c^k \nu^{\overline{N}(k)} \varphi(\nu^{-\overline{N}(k)})
   = C' c^k \psi(\nu^{-\overline{N}(k)}).
  \end{equation}
  We are only left with showing that the rightmost-hand side
  of~\eqref{eq:HausdorffEstimate} tends to zero as $k \to \infty.$
  Fix some integer $l \geq 2$ and consider $k_{l-1} < k \leq k_l.$
  It could be that for a particular value of $l,$ there is no $k$ in this range
  since it could happen that $k_{l-1} = k_l.$
  However, since the sequences $\{N_l\}$ and $\{k_l\}$ both tend to infinity,
  there are infinitely many values of $l$ for which $k_{l-1} < k_l.$
  Thus, we can assume that there exists $k_{l-1} < k \leq k_l.$
  For such values of $k,$ we have that $N_{l-1} < \overline{N}(k) \leq N_l,$
  so condition~\eqref{eq:RangeToBoundPsi} and the definition of $t_l$ gives that
  \begin{equation*}
   \psi(\nu^{-\overline{N}(k)}) \leq \nu^{\overline{N}(k)\nu^{-3l}}.
  \end{equation*}
  Moreover, for this range of values of $k$ we can use~\eqref{eq:NToVarianceEstimate}
  to get that
  \begin{equation}
   \label{eq:PsiEstimate}
   \psi(\nu^{-\overline{N}(k)}) \leq \nu^{2(m_0+l+1)\nu^{-l}k/c_0^2}.
  \end{equation}
  Therefore, inserting estimate~\eqref{eq:PsiEstimate} in~\eqref{eq:HausdorffEstimate}
  and letting $k \to \infty,$ so that $l \to \infty$ as well,
  concludes the proof.
 \end{proof}

    \printbibliography

    \Addresses

\end{document}